\documentclass[11pt,a4paper]{article}
\usepackage{aligned-overset}
\usepackage{amsmath,amsthm}
\usepackage{authblk}
\usepackage[backend=bibtex,maxbibnames=100,bibencoding=utf8,giveninits=true]{biblatex}
\usepackage{cancel}
\usepackage[hmargin={26mm,26mm},vmargin={30mm,35mm}]{geometry}
\usepackage{graphicx}
\usepackage{hyperref}
\usepackage{newtxtext}
\usepackage{newtxmath}
\usepackage[bb=stix]{mathalpha}
\usepackage{tikz-cd}
\usepackage{xcolor}
\usepackage{comment}
\usepackage{float}

\bibliography{pat}

\DeclareGraphicsExtensions{.pdf,.png,.jpg,.jpeg}
\graphicspath{{figures/}}

\newcommand{\email}[1]{\href{mailto:#1}{#1}}

\theoremstyle{plain}
\newtheorem{theorem}{Theorem}
\newtheorem{proposition}[theorem]{Proposition}
\newtheorem{lemma}[theorem]{Lemma}
\newtheorem{corollary}[theorem]{Corollary}
\newtheorem{assumption}[theorem]{Assumption}

\theoremstyle{remark}
\newtheorem{remark}[theorem]{Remark}

\numberwithin{theorem}{section}

\newcommand{\Real}{\mathbb{R}}
\newcommand{\Natural}{\mathbb{N}}

\newcommand{\dofmap}{\sigma}

\newcommand{\virtual}[1]{\mathrm{#1}}
\newcommand{\discrete}[1]{\mathbb{#1}}

\newcommand{\norm}[2]{\|#2\|_{#1}}
\newcommand{\seminorm}[2]{|#2|_{#1}}

\newcommand{\tnorm}[2]{|\!|\!| #2|\!|\!|_{#1}}

\newcommand{\Th}{\mathcal{T}_h}
\newcommand{\Sh}{\mathcal{S}_h}
\newcommand{\ST}{\mathcal{S}_T}

\newcommand{\Eh}{\mathcal{E}_h}
\newcommand{\Vh}{\mathcal{V}_h}

\newcommand{\ET}{\mathcal{E}_T}

\newcommand{\VT}{\mathcal{V}_T}
\newcommand{\VE}{\mathcal{V}_E}
\newcommand{\skel}{\Gamma_h}

\newcommand{\Poly}[1]{\mathcal{P}^{#1}}

\newcommand{\lproj}[2]{\pi_{#1}^{#2}}

\newcommand{\RTL}{R_T^L}
\newcommand{\RTa}{R_T^a}

\newcommand{\trrec}[1]{\gamma_{#1}}

\begin{document}

\title{Key challenges and bridges among convergence analysis techniques for polytopal methods}
\author[1]{Louren\c{c}o Beir\~{a}o da Veiga}
\author[2]{Daniele A. Di Pietro}
\author[2,3]{J\'er\^ome Droniou}
\affil[1]{%
  Dipartimento di Matematica e Applicazioni, Università degli Studi di Milano-Bicocca,
  Piazza dell’Ateneo Nuovo 1, 20126 Milano, Italy,
  \email{lourenco.beirao@unimib.it}
}
\affil[2]{%
  IMAG, Universit\'e de Montpellier, CNRS, Montpellier 34090, France\\
  \email{daniele.di-pietro@umontpellier.fr},
  \email{jerome.droniou@cnrs.fr}
}
\affil[3]{%
  School of Mathematics, Monash University, Melbourne, Australia
}
\maketitle

\begin{abstract}
  Polytopal methods provide a flexible framework for the numerical approximation of partial differential equations on general meshes.
  Their convergence analysis raises specific challenges due to their inherently non-conforming nature and, in many cases, the fully discrete nature of their solution.
  Two main techniques are considered: the virtual-function approach, used, e.g., in the context of Virtual Element Methods, and the fully discrete approach, which underlies, e.g., the Discrete de Rham method.
  We introduce here a novel framework based on the notion of conforming liftings, namely bounded and consistent mappings from the discrete space into the continuous space.
  This approach bridges the virtual and fully discrete viewpoints, clarifies the role of norm equivalence for virtual functions, and leads to a decomposition of the consistency error usable for polytopal methods.
  The three approaches are demonstrated on a model problem, which provides the opportunity to discuss relevant technical points.
  Bridges with the convergence properties of discrete differential complexes are also built.
  \smallskip\\
  \textbf{MSC2020:} 
  65N12, 
  65N15, 
  65N30, 
  65N08  
  \smallskip\\
  \textbf{Key words:} polytopal methods, convergence analysis, virtual element methods, discrete de Rham methods, discrete differential complexes
\end{abstract}



\section{Introduction}

Polytopal methods emerged in the late 2000s as a new technology for the solution of partial differential equations (PDEs) on more general meshes than those supported by classical finite elements.
In this work we discuss analysis techniques for polytopal methods with the twofold goal of 1) building bridges among them and with the literature on classical finite elements and 2) highlighting specific difficulties.
\smallskip

A key difference between polytopal and finite element methods lies in the fact that the former are inherently non-conforming, either because projections are used in the definition of discrete forms, as in Virtual Element methods (VEM) \cite{Beirao-da-Veiga.Brezzi.ea:13,Beirao-da-Veiga.Brezzi.ea:14,Beirao-da-Veiga.Brezzi.ea:23}, or because both spaces and forms are fully discrete, as in Discrete de Rham (DDR) methods \cite{Di-Pietro.Droniou.ea:20,Di-Pietro.Droniou:23,Bonaldi.Di-Pietro.ea:25}.
This difference is reflected in specific difficulties in their analysis.

The classical convergence analysis of finite element methods  for linear PDEs involves estimating the error $u - u_h$ between the exact solution $u$ and its numerical approximation $u_h$.
Classical approaches for fully conforming methods, such as Céa's Lemma \cite{Cea:64}, estimate this error in one stroke.
This approach is less natural in a non-conforming setting, where the proof usually starts by distinguishing, through a triangle inequality, the approximation and consistency contributions; see the classical paper \cite{Strang:72} and also \cite{Ern.Guermond:21*1} for more recent approaches.
Several error estimates can then be derived depending on how the consistency error is further decomposed.

In the context of polytopal methods, however, the difference $u - u_h$ may not make sense.
This is the case for methods in fully discrete formulation such as, e.g.,
DDR, but also
Mimetic Finite Difference \cite{Kuznetsov.Lipnikov.ea:04,Brezzi.Lipnikov.ea:05,Beirao-da-Veiga.Lipnikov.ea:11,Beirao-da-Veiga.Lipnikov.ea:14},
Finite Volume \cite{Droniou.Eymard:06,Droniou.Eymard.ea:10},
and Hybrid High-Order \cite{Di-Pietro.Ern.ea:14,Di-Pietro.Ern:15,Di-Pietro.Droniou:20} methods.
In this situation, the main focus is on the consistency error component $u_h - I_h u$, with $I_h u$ denoting a suitable approximation (obtained, e.g., via interpolation) of $u$ in the discrete space.
In \cite{Di-Pietro.Droniou:18}, a framework for a fully discrete error analysis encompassing estimates of $u_h - I_h u$ in both the natural stability norm and in weaker norms is proposed.
A major difference with respect to classical finite element estimates is that no attempt is made to further decompose the consistency error, as the latter can often be estimated by direct manipulations.
An important exception is encountered when considering the DDR discretization of problems involving curl-curl terms; see \cite[Section 6.5]{Di-Pietro.Droniou:23}.

Even when the difference $u - u_h$ makes sense, as in VEM (where $u_h$ denotes the virtual function corresponding to the solution to the scheme), additional difficulties appear with respect to finite elements.
In \cite{Beirao-da-Veiga.Brezzi.ea:13} it was shown that, by a simple argument, the error $u_h - I_h u$ (here evaluated in a continuous norm) can be bounded by a polynomial approximation error plus an interpolation error in the virtual space.
This abstract approach assumes that the virtual discretization satisfies certain key properties, whose proofs require the development of new tools beyond those available in the finite element literature (for instance, due to the fact that VEM does not rely on a reference element).
We mention, in particular, the requirements on stabilization bilinear forms and those concerning the approximation properties for virtual spaces; see, e.g., \cite{Beirao-da-Veiga.Lovadina.ea:17, Brenner.Sung:18,Chen.Huang:18,Beirao-da-Veiga.Mascotto.ea:22}.
We will show here that these points are actually closely related through the equivalence of the norm of a virtual function and that of its degrees of freedom (DOFs).
In passing, defining the error as $u - u_h$ involves comparing $u$ to the virtual function $u_h$, which is not the output of the actual computation.
\smallskip

In this paper, we discuss the two main analysis approaches for polytopal methods (based on virtual functions and fully discrete) and bridge them through a new one based on the notion of \emph{conforming lifting}.
The latter is a bounded and consistent mapping of the discrete space into the continuous space which, unlike virtual functions, is not necessarily implicitly defined as the solution to a local PDE problem, and may be fully computable.
This new framework explicitly shows the role of the above-mentioned norm equivalence for virtual functions and leads to a decomposition of the consistency error, in the spirit of finite elements estimates, that is usable for polytopal methods; see \eqref{eq:discrete:Eh:decomposition}.
\smallskip

The idea of conforming lifting is not entirely new, and has actually already been used in the context of non-conforming finite element methods, e.g., to analyze multigrid solvers \cite{Brenner:99}, or for the convergence analysis for fourth-order linear or nonlinear PDEs such as Kirchoff--Love plates equations \cite{Brenner.Sung.ea:13} and von Karm\`an equations \cite{Mallik.Nataraj:16}.
Conforming reconstructions also appear in the a posteriori error analysis of both conforming and non-conforming finite element methods; see, e.g., \cite{Achdou.Bernardi.ea:03,Karakashian.Pascal:03,Ern.Vohralik:15,Carstensen.Puttkammer:20,Ern.Vohralik:20} as well as \cite[Chapter 7]{Vohralik:24}.
In the context of polytopal schemes, conforming liftings have been used as a critical tool in the convergence analysis of Mimetic Finite Difference methods, see for instance \cite{Brezzi.Buffa.ea:09,Beirao-da-Veiga.Lipnikov.ea:11,Beirao-da-Veiga.Lipnikov.ea:14}.
They have also been used to prove Korn inequalities for polytopal schemes in linear elasticity (see \cite[Lemma 1]{Botti.Di-Pietro.ea:19} and also \cite[Section 7.2.3]{Di-Pietro.Droniou:20} and \cite{Droniou.Haidar.ea:25}).
In all the cases above, the liftings are conforming in the continuous space and typically constructed by making use of a sub-triangulation in such a way to remain close to the lifted function.
By contrast, the lifting used in \cite{Di-Pietro.Droniou:22} for a DDR discretization of Reissner--Mindlin plates problem additionally needs to fulfill fine projection properties, similar to \eqref{eq:Lh:consistency:a} below, to ensure robustness.

More recently \cite{Di-Pietro.Droniou.ea:25}, a consistency and convergence analysis based on conforming liftings has made it possible to generalize the results of \cite{Di-Pietro.Droniou:23} to the DDR complex of differential forms \cite{Bonaldi.Di-Pietro.ea:25}.
These encouraging results have served as a motivation for the present work, which aims to extend them to more general situations.
\smallskip

The rest of this work is organized as follows.
In Section \ref{sec:framework} we first revisit the error analysis based on virtual functions and continuous norms, the natural one for VEM, and then present a new framework based on the notion of conforming lifting.
The role of the norm equivalence for virtual functions is discussed.
In Section \ref{sec:application} we apply these analysis approaches to a nodal scheme for the Poisson equation, and additionally show that the consistency error can be estimated by direct manipulations, as often done for DDR schemes.
Appendix \ref{app:scaled.trace} details a trace lemma used in the VEM analysis and, finally, in Appendix \ref{sec:adjoint.consistency}, we bridge the notion of consistency error to that of adjoint consistency appearing in the analysis of DDR complexes.
We specifically identify this notion with that of consistency for the co-derivative operator, and thus establish a clear link with the inspiring work \cite{Di-Pietro.Droniou.ea:25} and also with recent work on Finite Element Systems \cite{Christiansen.Rapetti:25}.


\section{An abstract framework}\label{sec:framework}

\subsection{Setting}\label{sec:abstract.framework:setting}

Let $V$ be a Hilbert space with inner product-induced norm $\norm{V}{\cdot}$.
Let $a : V \times V \to \Real$ be a continuous bilinear form and $\ell : V \to \Real$ be a continuous linear form.
We aim to approximate the problem:
Find $u \in V$ such that
\begin{equation}\label{eq:weak}
  a(u,v) = \ell(v)
  \qquad \forall v \in V .
\end{equation}
This problem is henceforth assumed to be well-posed.
Denote by $V_h$ a vector space with inner product-induced norm $\norm{V_h}{\cdot}$.
Here, the subscript $h$ refers to a parameter (typically, the mesh size) that we let tend to zero.
Let $a_h : V_h\times V_h \to \Real$ be a bilinear form and $\ell_h : V_h \to \Real$ a linear form.
We consider the approximation of \eqref{eq:weak} corresponding to the discrete problem:
Find $u_h \in V_h$ such that
\begin{equation}\label{eq:abstract:discrete}
  a_h(u_h,v_h) = \ell_h(v_h)
  \qquad \forall v_h \in V_h.
\end{equation}

In general, the difference $u - u_h$ may not make sense as, e.g., the space $V_h$ could be spanned by objects of entirely different nature with respect to $V$.
For instance, in the fully discrete method considered in Section \ref{sec:abstract.framework:fully.discrete}, the present framework is applied with $V_h$ spanned by vectors of polynomials.
In order to define an approximation error, we therefore introduce an interpolator $I_h : V_I \to V_h$ defined on a subspace $V_I \subset V$ possibly featuring additional regularity, and define the \emph{approximation error} as $u_h - I_h u$.
Notice that, from this point on, we will tacitly assume that $u \in V_I$, so that $I_h u$ makes sense.

\begin{remark}[Interpolator]\label{rem:gen-int}
  At this stage, the notion of interpolator is quite general, and $I_h u$ can be understood as any  approximation of $u$ in $V_h$.
  In the following sections, both in order to simplify the presentation and draw a clearer bridge among the continuous and discrete approaches, we will identify $I_h u$ with the classical interpolator defined from DOF values; see \eqref{eq:virtual:Ih}.
  The presented analysis can be generalized to different choices of $I_h u$, for instance inspired by Cl\'ement, Scott--Zhang or Oswald approximation procedures, or by bounded cochain projections encountered in finite element complexes (see, e.g., \cite{Arnold.Falk.ea:06,Ern.Guermond:16}).
\end{remark}

Next, following \cite{Di-Pietro.Droniou:18}, we define the \emph{consistency error} as the linear form $\mathcal{E}_h : V_h \to \Real$ such that
\begin{equation}\label{eq:consistency.error}
  \mathcal{E}_h(v_h)
  \coloneqq \ell_h(v_h) - a_h(I_h u, v_h)
  \qquad \forall v_h \in V_h.
\end{equation}
Similar definitions of the consistency error are also used for the analysis of non-conforming finite element methods; see, e.g., \cite[Eq. (27.3)]{Ern.Guermond:21*1}.

\begin{remark}[Consistency error]\label{rem:consistency.error}
  When \eqref{eq:weak} corresponds to the weak formulation of a PDE problem, obtained using integration by parts, the consistency error measures the extent to which the corresponding integration by parts formula fails to hold at the discrete level.
  Consider, e.g., the Poisson problem with homogeneous Dirichlet boundary conditions and forcing term $f \in L^2(\Omega)$, corresponding to the following choices:
  \begin{equation}\label{eq:poisson}
    V = H_0^1(\Omega),\qquad
    a(w,v) = \int_\Omega \nabla u \cdot \nabla v,\qquad
    \ell(v) = \int_\Omega f v.
  \end{equation}
  Accounting for the fact that $f = - \Delta u$ almost everywhere in $\Omega$, it is clear that the consistency error $\mathcal{E}_h(v_h)$ measures by how much the integration by parts
  \[
  - \int_\Omega \Delta u \, v + \int_\Omega \nabla u \cdot \nabla v = 0
  \]
  fails to hold when the discrete space and forms replace the continuous ones.
\end{remark}

The following result shows that, if the discrete problem \eqref{eq:abstract:discrete} is well-posed, the approximation error is bounded by the consistency error.
Upon assuming uniform boundedness for $a_h$, this estimate is quasi-optimal \cite[Remark A.8]{Di-Pietro.Droniou:20}.

\begin{theorem}[Estimate of the approximation error]\label{thm:error.estimate}
  Assume that there exists $\alpha > 0$ such that
  \[
  \alpha \norm{V_h}{w_h}
  \le \sup_{v_h \in V_h \setminus \{0\}} \frac{a_h(w_h, v_h)}{\norm{V_h}{v_h}} \qquad \forall w_h\in V_h.
  \]
  Then, it holds
  \[
  \norm{V_h}{u_h - I_h u}
  \le \alpha^{-1} \norm{V_h'}{\mathcal{E}_h}.
  \]
\end{theorem}

\begin{proof}
  See \cite[Theorem 10]{Di-Pietro.Droniou:18}.
\end{proof}

\begin{corollary}[Estimate of the approximation error for conforming discrete spaces]\label{cor:conforming.case}
  Under the assumptions of Theorem \ref{thm:error.estimate}, and further assuming $V_h \subset V$ and $\norm{V_h}{\cdot}=\norm{V}{\cdot}$, it holds
  \begin{equation}\label{eq:conforming.case:error.estimate}
    \norm{V}{u - u_h}
    \le \norm{V}{u - I_h u} + \norm{V}{u_h - I_h u}
    \le \norm{V}{u - I_h u} + \alpha^{-1} \norm{V'}{\mathcal{E}_h}.
  \end{equation}
\end{corollary}

\subsection{Convergence analysis based on virtual functions}\label{sec:abstract.framework:vem}

We consider in this section a VEM approximation of problem \eqref{eq:weak}.
To simplify the identification of virtual objects (spaces, interpolators, functions, bilinear forms), we denote them with upright Roman fonts: $\virtual{V}_h$, $\virtual{I}_h$, $\virtual{v}_h$, $\virtual{a}_h$, etc.

We assume, from this point on, that there exists a Hilbert space $L$ with inner product $(\cdot,\cdot)_L$ and induced norm $\norm{L}{\cdot}$ such that $V$ injects continuously into $L$, and that there exists $f \in L$ such that
\begin{equation}\label{eq:rel.ell.L}
  \ell(v) = (f, v)_L \qquad \forall v \in L.
\end{equation}
We additionally assume that our continuous and discrete spaces and operators can be restricted as local spaces and operators to each element $T$ of a finite set $\Th$, and that these local objects fully reconstruct the global ones.
Typically, for a PDE problem set on a domain $\Omega$, $\Th$ will be a partition of $\Omega$ into elements $T$ (and $h$ will be the maximum diameter of these elements).
Restrictions to $T \in \Th$ of discrete objects are denoted replacing the subscript $h$ with $T$.

\subsubsection{Discrete problem}

Let $\virtual{V}_h \subset V_I$ be a VEM space with DOFs $\dofmap_h : V_I \to \Real^{\dim(V_h)}$.
As usual in finite and virtual elements, the interpolator $\virtual{I}_h : V_I \to \virtual{V}_h$ is defined by the following condition:
\begin{equation}\label{eq:virtual:Ih}
  \dofmap_h \virtual{I}_h v = \dofmap_h v \qquad \forall v \in V_I.
\end{equation}
The fact that $\virtual{I}_h$ is uniquely defined by \eqref{eq:virtual:Ih} is a consequence of the unisolvence of the DOFs.

We consider a bilinear form $\virtual{a}_h : \virtual{V}_h \times \virtual{V}_h \to \Real$ such that, for all $(\virtual{w}_h, \virtual{v}_h) \in \virtual{V}_h \times \virtual{V}_h$,
\[
\virtual{a}_h(\virtual{w}_h, \virtual{v}_h)
\coloneqq \sum_{T \in \Th} \virtual{a}_T(\virtual{w}_T, \virtual{v}_T),
\]
where, for all $T\in\Th$, $\virtual{a}_T : \virtual{V}_T \times \virtual{V}_T \to \Real$ is a local bilinear form satisfying the following consistency property:
For some finite-dimensional (typically polynomial) local space $\mathcal{P}_T \subset \virtual{V}_T$,
\begin{equation}\label{eq:virtual:consistency}
  \virtual{a}_T(w, \virtual{v}_T)
  = a_{|T}(w, \virtual{v}_T)
  \qquad \forall (w, \virtual{v}_T) \in \mathcal{P}_T \times \virtual{V}_T.
\end{equation}

Let $\widetilde{\mathcal{P}}_T \subset \mathcal{P}_T$ and denote by $\Pi_T^L$ the $L_{|T}$-orthogonal projector onto $\widetilde{\mathcal{P}}_T$.
We discretize the right-hand side by means of the linear form $\ell_h : \virtual{V}_h \to \Real$ such that
\[
\ell_h(\virtual{v}_h) \coloneqq \sum_{T \in \Th} (f, \Pi_T^L \virtual{v}_T)_{L_{|T}}.
\]
 The VEM scheme based on the conforming space $\virtual{V}_h$ reads:
Find $\virtual{u}_h \in \virtual{V}_h$ such that
\begin{equation}\label{eq:virtual:discrete}
  \virtual{a}_h(\virtual{u}_h, \virtual{v}_h)
  = \ell_h(\virtual{v}_h) \qquad \forall \virtual{v}_h \in \virtual{V}_h.
\end{equation}

\begin{remark}[Conforming Galerkin setting]\label{rem:galerkin}
If $\widetilde{\mathcal P}_T=\mathcal P_T=\virtual{V}_T$ for all $T\in\Th$, then \eqref{eq:virtual:consistency} enforces $\virtual{a}_h=a$ on $\virtual{V}_h\times\virtual{V}_h$, and we have $\Pi_T^L=\mathrm{Id}$ on $\virtual{V}_T$ (for all $T\in\Th$) which gives $\ell_h=\ell$ on $\virtual{V}_h$. The scheme \eqref{eq:virtual:discrete} is then the conforming Galerkin scheme for \eqref{eq:weak} based on the space $\virtual{V}_h$. 
\end{remark}

\begin{remark}[Consistency and non-conformity]\label{rem:cons}
    Following up on Remark \ref{rem:galerkin}, in the generic VEM setting (i.e., on general meshes), the virtual space $\virtual{V}_T$ is strictly larger than $\mathcal P_T$, the consistency \eqref{eq:virtual:consistency} can only be imposed on $\mathcal P_T\times\virtual{V}_T\subsetneq \virtual{V}_T\times\virtual{V}_T$, and $\virtual{a}_h$ does not coincide with $a$ on $\virtual{V}_h\times\virtual{V}_h$. This is due to the requirement of defining $\virtual{a}_h$ in a way that is both stable and \emph{computable} from the DOFs. The same holds for $\ell_h$, which cannot, in general, be taken equal to $\ell$. Thus, despite the fact that $\virtual{V}_h \subset V$, the scheme \eqref{eq:virtual:discrete} is non-conforming.

  We finally observe that condition \eqref{eq:virtual:consistency} can be relaxed, for instance in presence of some variable coefficients in the equation, and required to hold up to some approximation term of $w$, see for instance \cite{Beirao-da-Veiga.Brezzi.ea:16,Cangiani.Manzini.ea:17}.
\end{remark}

\subsubsection{Error analysis}\label{sec:abstract.framework:vem.error.analysis}

From this point on, inequalities of the form $a \le C b$ with $C > 0$ independent of $a$, $b$, $h$ and $T\in\Th$ (for local quantities) are abbreviated as $a \lesssim b$, and we use $a \simeq b$ as a shortcut for ``$a \lesssim b$ and $b \lesssim a$''.
We will also require the following inf-sup condition and local boundedness property on the bilinear form $\virtual{a}_h$:
\begin{alignat}{2}
  \label{eq:virtual:inf-sup}
  \norm{V}{\virtual{w}_h}
  &\lesssim \sup_{\virtual{v}_h \in \virtual{V}_h \setminus \{ \virtual{0} \}} \frac{\virtual{a}_h(\virtual{w}_h,\virtual{v}_h)}{\norm{V}{\virtual{v}_h}}&&\qquad
  \forall \virtual{w}_h \in \virtual{V}_h,\\
  \label{eq:virtual:local-bound}
  \virtual{a}_T(\virtual{w}_T, \virtual{v}_T)
  &\lesssim \norm{V_{|T}}{\virtual{w}_T} \norm{V_{|T}}{\virtual{v}_T}&&\qquad\forall T\in\Th\,,\quad\forall (\virtual{w}_T, \virtual{v}_T) \in \virtual{V}_T \times \virtual{V}_T.
\end{alignat}

\begin{remark}[Norm on the virtual space $\virtual{V}_h$]
  Note that these inf-sup and boundedness properties of the virtual bilinear form are stated in terms of the norm on $V$, not on some other discrete (and computable) norm specifically defined on $\virtual{V}_h$; the latter option is linked to the discussion in Section \ref{sec:abstract.framework:fully.discrete}. The choice of endowing $\virtual{V}_h$ with $\norm{V}{\cdot}$ is natural since $\virtual{V}_h\subset V$, but it can make \eqref{eq:virtual:inf-sup}--\eqref{eq:virtual:local-bound} quite challenging to establish, see in particular Section \ref{sec:nodal.virtual.analysis} and \cite{Brenner.Guan.ea:17,Beirao-da-Veiga.Lovadina.ea:17,Beirao-da-Veiga.Mascotto.ea:22}.
\end{remark}

\begin{theorem}[Error estimate based on virtual functions]\label{thm:vem:error.estimate}
  Denote by $u$ the solution to \eqref{eq:weak} and assume \eqref{eq:virtual:consistency}, \eqref{eq:virtual:inf-sup}, and \eqref{eq:virtual:local-bound}.
  Then, it holds,
  \begin{equation}\label{eq:vem:error.estimate}
    \norm{V}{u - u_h}
    \lesssim \norm{V}{u - \virtual{I}_h u}
    + \left[
      \sum_{T \in \Th} \left(
      \inf_{w\in\mathcal P_T}\norm{V_{|T}}{u - w}^2
      + \norm{L_{|T}}{f - \Pi_T^L f}^2
      \right)
      \right]^{\frac12}.
  \end{equation}
\end{theorem}

\begin{proof}
  Since we have assumed \eqref{eq:virtual:inf-sup} and $\virtual{V}_h \subset V$, the estimate \eqref{eq:conforming.case:error.estimate} holds.
  It only remains to bound $\norm{V'}{\mathcal{E}_h}$.
  We start from the following decomposition: %
  For all $\virtual{v}_h \in \virtual{V}_h$,
  \begin{equation}\label{eq:virtual:Eh:decomposition}
    \mathcal{E}_h(\virtual{v}_h)
    \overset{\eqref{eq:weak}}= \underbrace{%
      \ell_h(\virtual{v}_h) - \ell(\virtual{v}_h)
    }_{\eqcolon \mathcal{E}_{\ell,h}}
    + \underbrace{%
      a(u, \virtual{v}_h) - \virtual{a}_h(\virtual{I}_h u, \virtual{v}_h).
    }_{\eqcolon \mathcal{E}_{a,h}}
  \end{equation}

  For all $T\in\Th$ and all $w\in\mathcal P_T$, we have
  \[
  \begin{aligned}
    a_{|T}(u, \virtual{v}_T) - \virtual{a}_T(\virtual{I}_T u, \virtual{v}_T)
    \overset{\eqref{eq:virtual:consistency}}&=
    a_{|T}(u - w, \virtual{v}_T)
    + \virtual{a}_T(w - \virtual{I}_T u, \virtual{v}_T)
    \\
    \overset{\eqref{eq:virtual:local-bound}}&\lesssim
    \left(
    \norm{V_{|T}}{u - w}^2
    + \norm{V_{|T}}{w - \virtual{I}_T u}^2
    \right)^{\frac12}
    \norm{V_{|T}}{\virtual{v}_T},
  \end{aligned}
  \]
  where we have additionally used the $V_{|T}$-boundedness of $a_{|T}$ in the second step.
  Taking the infimum over $w\in\mathcal P_T$ for each $T \in \Th$, summing the resulting inequality over the mesh elements, and using a Cauchy--Schwarz inequality, we infer
  \begin{equation}\label{eq:virtual:Eah:estimate.2}
    \mathcal{E}_{a,h} = \sum_{T \in \Th} \left[
      a_{|T}(u, \virtual{v}_T) - \virtual{a}_T(\virtual{I}_T u, \virtual{v}_T)
      \right]
    \lesssim
    \left(
    \norm{V}{u - \virtual{I}_h u}^2
    + \sum_{T\in\Th} \inf_{w\in\mathcal P_T}\norm{V_{|T}}{u - w}^2
    \right)^{\frac12}\norm{V}{\virtual{v}_h}.
  \end{equation}

  To bound $\mathcal{E}_{\ell,h}$, we simply use the fact that $\Pi_T^L - {\rm Id}$ is self-adjoint for the inner product of $L_{|T}$ to write
  \begin{equation}\label{eq:virtual:Elh:estimate}
    \begin{aligned}
      \mathcal{E}_{\ell,h}
      &= \sum_{T \in \Th} (f, \Pi_T^L \virtual{v}_T - \virtual{v}_T)_{L_{|T}}
      = \sum_{T \in \Th} (\Pi_T^L f - f, \virtual{v}_T)_{L_{|T}}
      \\
      &\le \left(
      \sum_{T \in \Th} \norm{L_{|T}}{f - \Pi_T^L f}^2
      \right)^{\frac12} \norm{L}{\virtual{v}_h}
      \lesssim \left(
      \sum_{T \in \Th} \norm{L_{|T}}{f - \Pi_T^L f}^2
      \right)^{\frac12} \norm{V}{\virtual{v}_h},
    \end{aligned}
  \end{equation}
  where the last bound is justified by the fact that the injection $V \hookrightarrow L$ is continuous.

  Plugging \eqref{eq:virtual:Eah:estimate.2} and \eqref{eq:virtual:Elh:estimate} into \eqref{eq:virtual:Eh:decomposition} written for $\virtual{v}_h \neq \virtual{0}$,
  dividing by $\norm{V}{\virtual{v}_h}$,
  and passing to the supremum over $\virtual{v}_h \in \virtual{V}_h \setminus \{ \virtual{0} \}$, we infer that
  \begin{equation}\label{eq:vem.estimate.Eh}
  \norm{V'}{\mathcal{E}_h}
  \lesssim \left[
    \norm{V}{u - \virtual{I}_h u}^2
    + \sum_{T \in \Th} \left(
    \inf_{w\in\mathcal P_T}\norm{V_{|T}}{u - w}^2
    + \norm{L_{|T}}{f - \Pi_T^L f}^2
    \right)
    \right]^{\frac12}.
  \end{equation}
  Plugging this estimate into \eqref{eq:conforming.case:error.estimate} and using the inequality $a^2 + b^2 \le (a + b)^2$ valid for all $a,b \in \Real^+$, the desired result follows.
\end{proof}

\subsection{Convergence analysis based on a conforming lifting}\label{sec:abstract.framework:fully.discrete}

Let us now examine how things change when we adopt a fully discrete perspective, in which the numerical scheme is directly written in terms of the DOFs and the entire analysis is carried out using a discrete norm on the corresponding space.
In a similar way to virtual objects, we use blackboard fonts to help visually identify fully discrete objects: $\discrete{V}_h$, $\discrete{I}_h$, $\discrete{v}_h$, $\discrete{a}_h$, etc.

We endow the fully discrete space $\discrete{V}_h \coloneqq \dofmap_h\virtual{V}_h=\Real^{\dim(\virtual{V}_h)}$ with a norm $\norm{\discrete{V}_h}{\cdot}$.
For all $T\in\Th$, let $\RTL:\discrete{V}_T\to L_{|T}$ be a reconstruction operator, and define the discrete linear form $\ell_h : \discrete{V}_h \to \Real$ by setting
\[
\ell_h(\discrete{v}_h) \coloneqq \sum_{T \in \Th} (f, \RTL \discrete{v}_T)_{L_{|T}}.
\]
After designing a bilinear form $\discrete{a}_h : \discrete{V}_h \times \discrete{V}_h \to \Real$, the fully discrete scheme reads:
Find $\discrete{u}_h \in \discrete{V}_h$ such that
\begin{equation}\label{eq:fully:discrete}
  \discrete{a}_h(\discrete{u}_h, \discrete{v}_h) = \ell_h(\discrete{v}_h)
  \qquad \forall \discrete{v}_h \in \discrete{V}_h.
\end{equation}
As in the previous section, we assume that the global bilinear form $\discrete{a}_h$ is assembled element-wise from local bilinear forms defined on $\discrete{V}_T \times \discrete{V}_T$, i.e., for all $(\discrete{w}_h, \discrete{v}_h) \in \discrete{V}_h \times \discrete{V}_h$,
\begin{equation}\label{eq:discrete:ah}
  \discrete{a}_h(\discrete{w}_h, \discrete{v}_h)
  = \sum_{T \in \Th} \discrete{a}_T(\discrete{w}_T, \discrete{v}_T).
\end{equation}
We additionally assume the following inf-sup condition and local boundedness property:
\begin{alignat}{2}\label{eq:discrete:inf-sup}
  \norm{\discrete{V}_h}{\discrete{w}_h}
  &\lesssim \sup_{\discrete{v}_h \in \discrete{V}_h \setminus \{ \discrete{0} \}} \frac{\discrete{a}_h(\discrete{w}_h,\discrete{v}_h)}{\norm{\discrete{V}_h}{\discrete{v}_h}}
  &&\qquad\forall \discrete{w}_h\in\discrete{V}_h,\\
  \label{eq:discrete.a.T:boundedness}
  \discrete{a}_T(\discrete{w}_T, \discrete{v}_T)
  &\lesssim \norm{\discrete{V}_T}{\discrete{w}_T}
  \norm{\discrete{V}_T}{\discrete{v}_T}
  &&\qquad \forall T\in\Th\,,\quad \forall (\discrete{w}_T, \discrete{v}_T) \in \discrete{V}_T \times \discrete{V}_T.
\end{alignat}
By \eqref{eq:discrete:inf-sup}, the error estimate in Theorem \ref{thm:error.estimate} holds, and becomes meaningful if we can estimate the dual norm of the consistency error.

In many cases, it is possible to derive an estimate of $\norm{\discrete{V}_h'}{\mathcal{E}_h}$ by direct manipulations of its expression; see Section \ref{sec:application:direct.manipulations} below for an example among many others.
This approach has been historically favored for methods in fully discrete formulation, and can be particularly useful to design schemes which comply with key physical properties; see, e.g., \cite{Brezzi.Lipnikov.ea:05,Beirao-da-Veiga.Lipnikov.ea:11,Beirao-da-Veiga.Lipnikov.ea:14,Bonelle.Ern:14,Di-Pietro.Ern.ea:14,Di-Pietro.Ern:15,Boffi.Di-Pietro:18,Di-Pietro.Droniou:20,Droniou.Eymard.ea:18,Eymard.Gallouet.ea:00}.
Such manipulations, however, can only be performed on a case-by-case basis (possibly relying on some generic concepts, see \cite{Droniou.Eymard.ea:16} and \cite[Section 7]{Droniou.Eymard.ea:18}).
We use here a more original technique that expands the approach presented in Section \ref{sec:abstract.framework:vem} by making use of a conforming lifting of the DOFs, which is not necessarily a virtual function.
The properties required on such lifting are collected in the following assumption in which, for all $T\in\Th$, $\mathcal P_T$ is, as before, a finite-dimensional (usually polynomial) subspace of $(V_I)_{|T}$.
\begin{assumption}[Conforming lifting]\label{ass:conforming.lifting}
  There exists $L_h : \discrete{V}_h \to V$ such that the following properties hold:
  \begin{enumerate}
  \item \emph{Consistency.} For all $T \in \Th$,
    \begin{equation}\label{eq:Lh:consistency:a}
      \discrete{a}_T(\dofmap_T w, \discrete{v}_T)
      = a_{|T}(w, (L_h \discrete{v}_h)_{|T})\qquad\forall (w, \discrete{v}_h) \in \mathcal{P}_T \times \discrete{V}_h;
    \end{equation}

  \item \emph{Projection on $\widetilde{\mathcal P}_T$.} For all $T \in \Th$,
    \begin{equation}\label{eq:Lh:consistency:ell}
      \Pi_T^L (L_h \discrete{v}_h)_{|T} = \RTL \discrete{v}_T
      \qquad \forall \discrete{v}_T \in \discrete{V}_T;
    \end{equation}

  \item \emph{Boundedness.} For all $\discrete{v}_h \in \discrete{V}_h$,
    \begin{equation}\label{eq:Lh:boundedness}
      \norm{V}{L_h \discrete{v}_h}
      \lesssim \norm{\discrete{V}_h}{\discrete{v}_h}.
    \end{equation}
  \end{enumerate}
\end{assumption}

\begin{remark}[Locality of the lifting]
  Notice that we do not assume that the lifting is local, i.e., for a given $T \in \Th$, $(L_h \discrete{v}_h)_{|T}$ may depend on the entire $\discrete{v}_h$ and not only on $\discrete{v}_T$.
  This additional flexibility may be needed, e.g., when $V_h$ does not mimic the continuity properties of $V$, as is the case in HHO methods.
\end{remark}

\begin{theorem}[Error estimate based on a conforming lifting]\label{thm:discrete:error.estimate}
  Let Assumption \ref{ass:conforming.lifting} hold, and further assume \eqref{eq:discrete:inf-sup} and \eqref{eq:discrete.a.T:boundedness}.
  Then, it holds
  \begin{equation}\label{eq:discrete:error.estimate}
    \norm{\discrete{V}_h}{\discrete{u}_h - \dofmap_h u}
    \lesssim \left[
      \sum_{T \in \Th} \left(
      \inf_{w\in\mathcal P_T}\left(\norm{V_{|T}}{u - w}^2
      + \norm{\discrete{V}_T}{\dofmap_T(u - w)}^2\right)
      + \norm{L_{|T}}{f - \Pi_T^L f}^2
      \right)
      \right]^{\frac12}.
  \end{equation}
\end{theorem}

\begin{proof}
  Since we have assumed \eqref{eq:discrete:inf-sup}, by Theorem \ref{thm:error.estimate} we have the following estimate:
  \begin{equation}\label{eq:discrete:basic.error.estimate}
    \norm{\discrete{V}_h}{\discrete{u}_h - \dofmap_h u}
    \lesssim \alpha^{-1} \norm{\discrete{V}_h'}{\mathcal{E}_h}.
  \end{equation}
  We decompose the consistency error as follows:
  \begin{equation}\label{eq:discrete:Eh:decomposition}
    \mathcal{E}_h(\discrete{v}_h)
    = \underbrace{%
      \ell_h(\discrete{v}_h) - \ell(L_h \discrete{v}_h)
    }_{\mathcal{E}_{\ell,h}}
    + \underbrace{%
      a(u, L_h \discrete{v}_h) - \discrete{a}_h(\dofmap_h u, \discrete{v}_h).
    }_{\mathcal{E}_{a,h}}
  \end{equation}

  For all $T\in\Th$ and all $w\in\mathcal P_T$, it holds
  \[
  \begin{aligned}
    a_{|T}(u, (L_h \discrete{v}_h)_{|T})
    - \discrete{a}_T(\dofmap_T u, \discrete{v}_T)
    \overset{\eqref{eq:Lh:consistency:a}}&=
    a_{|T}(u-w, (L_h \discrete{v}_h)_{|T})
    + \discrete{a}_T(\dofmap_T (w-u), \discrete{v}_T)
    \\
    \overset{\eqref{eq:discrete.a.T:boundedness}}&\lesssim
    \norm{V_{|T}}{u - w} \norm{V_{|T}}{(L_h \discrete{v}_h)_{|T}}
    + \norm{\discrete{V}_T}{\dofmap_T (u - w)}
    \norm{\discrete{V}_T}{\discrete{v}_T}
    \\
    &\lesssim
    \left(\norm{V_{|T}}{u - w}^2+\norm{\discrete{V}_T}{\dofmap_T (u - w)}^2\right)^{\frac12}
    \left(\norm{V_{|T}}{(L_h \discrete{v}_h)_{|T}}^2+ \norm{\discrete{V}_T}{\discrete{v}_T}\right)^{\frac12},
  \end{aligned}
  \]
  where, in the second inequality, we have additionally used the continuity of $a_{|T}$ on $V_{|T}\times V_{|T}$.
  Taking the infimum over $w\in\mathcal P_T$, summing over $T\in\Th$ and using a Cauchy--Schwarz inequality together with \eqref{eq:Lh:boundedness}, we infer
  \begin{align}
    \mathcal{E}_{a,h}
    &= \sum_{T \in \Th} \left[
      a_{|T}(u, L_h \discrete{v}_h)
      - \discrete{a}_T(\dofmap_T u, \discrete{v}_T)
      \right]
    \nonumber\\
    &\lesssim \left[\sum_{T\in\Th}\inf_{w\in\mathcal P_T}\left(\norm{V_{|T}}{u - w}^2+\norm{\discrete{V}_T}{\dofmap_T (u - w)}^2\right)\right]^{\frac12}\norm{\discrete{V}_h}{\discrete{v}_h}.
    \label{eq:discrete:Eah:estimate}
  \end{align}

  For the error on the forcing term, we notice that, for all $T \in \Th$,
  \[
  (f, \RTL \discrete{v}_T-L_h \discrete{v}_h )_{L_{|T}}
  \overset{\eqref{eq:Lh:consistency:ell}}=
  (f, \Pi_T^L (L_h \discrete{v}_h)_{|T} -L_h \discrete{v}_h)_{L_{|T}}
  = (\Pi_T^L f -f, L_h \discrete{v}_h)_{L_{|T}},
  \]
  where the last step is justified by the fact that $\Pi_T^L - {\rm Id}$ is self-adjoint for the inner product of $L_{|T}$.
  Hence,
  \begin{equation}\label{eq:discrete:Elh:estimate}
    \mathcal{E}_{\ell,h}
    \lesssim \left(
    \sum_{T \in \Th} \norm{L_{|T}}{f - \Pi_T^L f}^2
    \right)^{\frac12} \norm{L}{L_h \discrete{v}_h}
    \lesssim \left(
    \sum_{T \in \Th} \norm{L_{|T}}{f - \Pi_T^L f}^2
    \right)^{\frac12} \norm{\discrete{V}_h}{\discrete{v}_h},
  \end{equation}
  where the conclusion follows using the continuity of the injection $V \hookrightarrow L$ followed by \eqref{eq:Lh:boundedness} to write $\norm{L}{L_h \discrete{v}_h} \lesssim \norm{V}{L_h \discrete{v}_h} \lesssim \norm{\discrete{V}_h}{\discrete{v}_h}$.

  Plugging \eqref{eq:discrete:Eah:estimate} and \eqref{eq:discrete:Elh:estimate} into \eqref{eq:discrete:Eh:decomposition} with $\discrete{v}_h \neq \discrete{0}$,
  dividing by $\norm{\discrete{V}_h}{\discrete{v}_h}$,
  and passing to the supremum over $\discrete{v}_h \in \discrete{V}_h \setminus \{ \discrete{0} \}$,
  we conclude that $\norm{\discrete{V}_h'}{\mathcal{E}_h}$ is bounded by the right-hand side of \eqref{eq:discrete:error.estimate}.
  Plugging this bound into \eqref{eq:discrete:basic.error.estimate}, the conclusion follows.
\end{proof}

\subsection{Discussion}

\subsubsection{Algebraic equivalence}

Since the DOFs mapping $\dofmap_h:\virtual{V}_h\to \discrete{V}_h$ is an isomorphism, the VEM scheme \eqref{eq:virtual:discrete} can be recast into the fully discrete formulation \eqref{eq:fully:discrete} setting
\begin{subequations}\label{eq:link:vem.fd}
  \begin{align}\label{eq:link:vem.fd.a}
    \discrete{a}_T(\discrete{v}_T,\discrete{w}_T)
    \coloneqq \virtual{a}_T(\dofmap_T^{-1}\discrete{v}_T,\dofmap_T^{-1}\discrete{w}_T)
    &\qquad \forall T \in \Th, \quad \forall (\discrete{v}_T,\discrete{w}_T) \in \discrete{V}_T \times \discrete{V}_T,
    \\ \label{eq:link:vem.fd.R}
    \RTL \coloneqq \Pi^L_T\circ \dofmap_T^{-1}&\qquad \forall T \in \Th.
  \end{align}
\end{subequations}

The main difference between the proofs of Theorems \ref{thm:discrete:error.estimate} and \ref{thm:vem:error.estimate} is that the conforming lifting $L_h \discrete{v}_h$ replaces the virtual function $\virtual{v}_h$ in the error decomposition; compare \eqref{eq:virtual:Eh:decomposition} with \eqref{eq:discrete:Eh:decomposition}.
Notice that, under the assumptions of Theorem \ref{thm:vem:error.estimate} and with the definitions \eqref{eq:link:vem.fd}, $L_h \coloneq \dofmap_h^{-1}$ defines a lifting that satisfies the first two properties in Assumption \ref{ass:conforming.lifting}:
\eqref{eq:Lh:consistency:a} follows from \eqref{eq:link:vem.fd.a} and \eqref{eq:virtual:consistency}, while \eqref{eq:Lh:consistency:ell} is just \eqref{eq:link:vem.fd.R}. The boundedness \eqref{eq:Lh:boundedness} is a more delicate matter that we discuss in Section \ref{sec:equivalence.norms}.
In practice, liftings $L_h$ satisfying Assumption \ref{ass:conforming.lifting} can be found without identifying a virtual space; see Section \ref{sec:fd.analysis.example}.

However, there is a subtle, yet important, difference between the estimates obtained with each approach: the estimate \eqref{eq:vem:error.estimate} includes the approximation error ${\norm{V}{u-\virtual{I}_hu}}$ of the virtual interpolator (which does not just come from the use of \eqref{eq:conforming.case:error.estimate}, but already appears in the bound \eqref{eq:vem.estimate.Eh} on $\norm{V'}{\mathcal E_h}$), while the estimate \eqref{eq:discrete:error.estimate} only relies on approximation properties of the local polynomial spaces. As shown in the next lemma, the latter can bound the former, \emph{provided the interpolator has boundedness properties}; see Remark \ref{rem:boundedness.Ih} regarding those.

\begin{lemma}[Approximation properties of the virtual interpolator]
  Endow $(V_I)_{|T}$ with a (semi-)norm $\norm{(V_I)_{|T}}{\cdot}$ and assume that the local interpolator $\virtual{I}_T:(V_I)_{|T}\to \virtual{V}_T$ satisfies the continuity bound
  \begin{equation}\label{eq:cont.IT}
    \norm{V_{|T}}{\virtual{I}_T v}\lesssim \norm{(V_I)_{|T}}{v}\qquad\forall v\in (V_I)_{|T}.
  \end{equation}
  Then, it holds
  \begin{equation*}  
    \norm{V_{|T}}{v-\virtual{I}_T v}
    \lesssim \inf_{w\in\mathcal P_T}\left(\norm{V_{|T}}{v - w}
    + \norm{(V_I)_{|T}}{v - w}
    \right)\qquad\forall v\in (V_I)_{|T}.
  \end{equation*}
\end{lemma}

\begin{proof}
  Take $v\in (V_I)_{|T}$ and $w\in\mathcal P_T$. Since $\mathcal P_T\subset\virtual{V}_T$, the definition \eqref{eq:virtual:Ih} of the interpolator shows that $\virtual{I}_Tw=w$, and thus
  \[
  \norm{V_{|T}}{v-\virtual{I}_Tv}=\norm{V_{|T}}{v-w-\virtual{I}_T(v-w)}\le \norm{V_{|T}}{v-w}+\norm{V_{|T}}{\virtual{I}_T(v-w)}
  \overset{\eqref{eq:cont.IT}}\lesssim \norm{V_{|T}}{v-w}+\norm{(V_I)_{|T}}{v-w}.
  \]
  Taking the infimum over $w \in \mathcal{P}_T$ concludes the proof.
\end{proof}

\subsubsection{The need for equivalence between discrete and continuous norms}\label{sec:equivalence.norms}

So far, we have not discussed the choice of the norm on $\discrete{V}_h$. Following \eqref{eq:link:vem.fd}, a possible choice would be
\begin{equation}\label{eq:discrete.norm.naive}
  \norm{\discrete{V}_h}{\discrete{v}_h}\coloneq \norm{V}{\dofmap_h^{-1}\discrete{v}_h}.
\end{equation}
Together with the choice $L_h=\dofmap^{-1}_h$, this would make the boundedness \eqref{eq:Lh:boundedness} trivial.
However, the choice \eqref{eq:discrete.norm.naive} has several drawbacks. First of all, in general, one cannot explicitly compute $\dofmap_h^{-1}\discrete{v}_h$, which means that $\norm{\discrete{V}_h}{\discrete{v}_h}$ would not be computable in practice (even though the vector $\discrete{v}_h$ itself is completely explicit!).
Second, in most situations the DOFs have specific meanings (e.g., polynomial moments on mesh entities) which drive more natural (and computable) choices for the norm on $\discrete{V}_h$, such as scaled $L^2$-norms of polynomials on mesh entities.

For these reasons, $\norm{\discrete{V}_h}{\cdot}$ is more often chosen independently of the norm on $V$. However, requiring the well-posedness of the virtual and fully discrete schemes  \eqref{eq:virtual:discrete}  and  \eqref{eq:fully:discrete} may demand a uniform (in $h$) equivalence between these norms, as we will see in what follows.

Instead of the most general inf-sup condition, let us consider the case where the stability results from coercivity.
Specifically, let us consider the following two sets of conditions:

\begin{itemize}
\item There are strictly positive real numbers $\alpha_V$ and $\beta_V$ independent of $h$ such that, for all $T\in\Th$,
  \begin{subequations}\label{eq:C1}
    \begin{alignat}{2}
      \alpha_V \norm{V_{|T}}{\virtual{v}_T}^2
      &\le \virtual{a}_T(\virtual{v}_T,\virtual{v}_T)
      &\quad& \forall \virtual{v}_T \in \virtual{V}_T,
      \label{eq:C1:coercivity}\\
      \virtual{a}_T(\virtual{w}_T, \virtual{v}_T)
      &\le \beta_V \norm{V_{|T}}{\virtual{w}_T} \norm{V_{|T}}{\virtual{v}_T}
      &\quad& \forall (\virtual{w}_T, \virtual{v}_T) \in \virtual{V}_T \times \virtual{V}_T;
      \label{eq:C1:boundedness}
    \end{alignat}
  \end{subequations}

\item There are strictly positive real numbers $\alpha_{\discrete{V}}$ and $\beta_{\discrete{V}}$ independent of $h$ such that, for all $T\in\Th$,
  \begin{subequations}\label{eq:C2}
    \begin{alignat}{2}\label{eq:C2:coercivity}
      \alpha_{\discrete{V}} \norm{\discrete{V}_T}{\discrete{v}_T}^2
      &\le \discrete{a}_T(\discrete{v}_T,\discrete{v}_T)
      &\quad& \forall \discrete{v}_T \in \discrete{V}_T,
      \\ \label{eq:C2:boundedness}
      \discrete{a}_T(\discrete{w}_T, \discrete{v}_T)
      &\le \beta_{\discrete{V}} \norm{\discrete{V}_T}{\discrete{w}_T} \norm{\discrete{V}_T}{\discrete{v}_T}
      &\quad& \forall (\discrete{w}_T, \discrete{v}_T) \in \discrete{V}_T \times \discrete{V}_T.
    \end{alignat}
  \end{subequations}
\end{itemize}
Notice, in passing, that the local maps $\norm{\discrete{V}_T}{\cdot}$ and $\norm{V_{|T}}{\cdot}$ that appear above could be seminorms instead of norms, but this bears no influence on the following steps.

\begin{proposition}[Stability conditions and norm equivalence]\label{prop:C1.C2.E}
  Let $\discrete{a}_h$ be defined by \eqref{eq:link:vem.fd.a}.
  If \eqref{eq:C1} and \eqref{eq:C2} hold, then the continuous and discrete norms are equivalent in the following sense:
  There exists a real number $\eta \ge 1$ independent of $h$ such that, for all $T \in \Th$,
  \begin{equation}\label{eq:E}
    \eta^{-1} \norm{\discrete{V}_T}{\dofmap_T \virtual{v}_T}
    \le \norm{V_{|T}}{\virtual{v}_T}
    \le \eta \norm{\discrete{V}_T}{\dofmap_T \virtual{v}_T}
    \qquad \forall \virtual{v}_T \in \virtual{V}_T.
  \end{equation}
  Moreover, if \eqref{eq:E} holds, then \eqref{eq:C1} and \eqref{eq:C2} are equivalent.
\end{proposition}

  \begin{remark}[Norm equivalence]
    In symbols, Proposition \ref{prop:C1.C2.E} states that
    \[
    \text{
      $\big(\eqref{eq:C1} \land \eqref{eq:C2}\big) \implies \eqref{eq:E}$\quad
      and\quad
      $\eqref{eq:E} \implies \big(
      \eqref{eq:C1} \iff \eqref{eq:C2}
      \big)$.
    }
    \]
    Notice that it does not seem possible to have the full equivalence $\eqref{eq:E} \iff \left( \eqref{eq:C1} \iff \eqref{eq:C2} \right)$ as it would require to prove, in addition, that $\left( \lnot\eqref{eq:C1} \land \lnot\eqref{eq:C2} \right) \implies \eqref{eq:E}$.
    Inferring $\big(\eqref{eq:C1} \land \eqref{eq:C2}\big) \implies \eqref{eq:E}$ when stability results from inf-sup conditions is also not entirely obvious.
  \end{remark}

\begin{proof}[Proof of Proposition \ref{prop:C1.C2.E}]
  Assume \eqref{eq:C1} and \eqref{eq:C2}. For all $T\in\Th$ and $\virtual{v}_T\in\virtual{V}_T$, write
  \[
  \norm{\discrete{V}_T}{\dofmap_T\discrete{v}_T}^2
  \overset{\eqref{eq:C2:coercivity}}\le \alpha_{\discrete{V}}^{-1}\discrete{a}_T(\dofmap_T\virtual{v}_T,\dofmap_T\virtual{v}_T)
  \overset{\eqref{eq:link:vem.fd.a}}=\alpha_{\discrete{V}}^{-1}\virtual{a}_T(\virtual{v}_T,\virtual{v}_T)
  \overset{\eqref{eq:C1:boundedness}}\le\alpha_{\discrete{V}}^{-1}\beta_V\norm{V_{|T}}{\virtual{v}_T}^2.
  \]
  Taking the square root concludes the proof of the first inequality in \eqref{eq:E}. The second is obtained the same way, reversing the roles of $\discrete{a}_T$ and $\virtual{a}_T$.

  Assume now that \eqref{eq:E} holds, and let us prove that \eqref{eq:C1} implies \eqref{eq:C2} (the converse follows by reversing the roles of the virtual and fully discrete forms and spaces).
  We can write, for all $T \in \Th$ and all $\discrete{v}_T \in \discrete{V}_T$,
  \[
  \norm{\discrete{V}_T}{\discrete{v}_T}^2
  \overset{\eqref{eq:E}}\le \eta^2 \norm{V_{|T}}{\dofmap_T^{-1}\discrete{v}_T}^2
  \overset{\eqref{eq:C1:coercivity}}\le \eta^2\alpha_V^{-1}\virtual{a}_T(\dofmap_T^{-1}\discrete{v}_T,\dofmap_T^{-1}\discrete{v}_T)
  \overset{\eqref{eq:link:vem.fd.a}}=\eta^2\alpha_V^{-1}\discrete{a}_T(\discrete{v}_T,\discrete{v}_T),
  \]
  which proves \eqref{eq:C2:coercivity} with $\alpha_{\discrete{V}} = \alpha_V \eta^{-2}$.
  The proof of \eqref{eq:C2:boundedness} is done in a similar way.
\end{proof}

\begin{remark}[Boundedness of the interpolator]\label{rem:boundedness.Ih}
  If the second inequality in the norm equivalence \eqref{eq:E} holds and, for all $T \in \Th$, the local DOFs map $\dofmap_T:(V_I)_{|T}\to \discrete{V}_T$ is continuous uniformly in $h$, i.e.,
  \begin{equation}\label{eq:cont.dofs}
    \norm{\discrete{V}_T}{\dofmap_T v}\lesssim \norm{(V_I)_{|T}}{v}
    \qquad \forall v\in (V_I)_{|T},
  \end{equation}
  then the interpolator satisfies the continuity property \eqref{eq:cont.IT}.
  Indeed, for all $v\in(V_I)_{|T}$, we can write
  \[
  \norm{V_{|T}}{\virtual{I}_Tv}\overset{\eqref{eq:E}}\lesssim \norm{\discrete{V}_T}{\dofmap_T\virtual{I}_Tv}
  \overset{\eqref{eq:virtual:Ih}}=\norm{\discrete{V}_T}{\dofmap_Tv}\overset{\eqref{eq:cont.dofs}}\lesssim
  \norm{(V_I)_{|T}}{v}.
  \]
  This observation, combined with the fact that a certain (polynomial) space $\mathcal{P}_T$ is contained in $V_T$, can be used to derive (local) VEM approximation properties for the interpolant $\virtual{I}_h u$; see for instance \cite{Beirao-da-Veiga.Mascotto.ea:22}.
\end{remark}

\subsection{Practical construction of a virtual bilinear form}\label{sec:practical.construction.vem}

Besides polynomial consistency \eqref{eq:virtual:consistency} and stability (which is assumed here to hold in the sense of coercivity, cf.~\eqref{eq:C1}), another criterion when defining the local virtual bilinear forms $\virtual{a}_T$, $T \in \Th$, is their \emph{computability}: they must be expressed using implementable formulas purely based on the knowledge of the DOFs of the virtual functions.
In practice, this is often achieved through the design of computable projections of the virtual functions.

Specifically, for any $T \in \Th$, and denoting as before by $\mathcal{P}_T \subset \virtual{V}_T$ a finite-dimensional (typically polynomial) local space, a virtual element approach relies on designing a suitable mapping $\Pi_T^a : \virtual{V}_T \to \mathcal{P}_T$ that satisfies
\begin{equation}\label{eq:Pi.polynomial}
  \Pi_T^a w = w\qquad\forall w\in \mathcal P_T
\end{equation}
and that is computable, in the sense that there exists an implementable operator $\RTa:\discrete{V}_T\to\mathcal P_T$ such that $\Pi_T^a=\RTa\circ\dofmap_T$.
Notice that, by \eqref{eq:Pi.polynomial}, $\Pi_T^a$ is a projector; see, e.g., \cite[Proposition 1.35]{Di-Pietro.Droniou:20}.
The local virtual bilinear form is then taken as
\[
\virtual{a}_T(\virtual{w}_T, \virtual{v}_T)
\coloneqq \tilde{\virtual{a}}_T(\virtual{w}_T,\virtual{v}_T)
+ \virtual{s}_T(\virtual{w}_T - \Pi_T^a \virtual{w}_T, \virtual{v}_T - \Pi_T^a \virtual{v}_T),
\]
where $\tilde{\virtual{a}}_T : \virtual{V}_T \times \virtual{V}_T \to \Real$ is a local bilinear form in charge of consistency, while
$\virtual{s}_T : \virtual{V}_T \times \virtual{V}_T \to \Real$ is a positive definite stabilization bilinear form.
By \eqref{eq:Pi.polynomial}, the consistency property \eqref{eq:virtual:consistency} translates into the condition
\begin{equation}\label{eq:virtual:tilde.a.T}
  \tilde{\virtual{a}}_T(w, \virtual{v}_T)
  = a_{|T}(w, \virtual{v}_T)
  \qquad \forall (w, \virtual{v}_T) \in \mathcal{P}_T \times \virtual{V}_T.
\end{equation}

\begin{remark}[Consistent bilinear form and projector]\label{rem:cons-bil}
  A typical choice for $\tilde{\virtual{a}}_T$, which ensures its computability, is
  \[
  \tilde{\virtual{a}}_T(\virtual{w}_T, \virtual{v}_T)=a_{|T}(\Pi_T^a\virtual{w}_T,\Pi_T^a\virtual{w}_T).
  \]
  By \eqref{eq:Pi.polynomial}, this choice satisfies the consistency \eqref{eq:virtual:tilde.a.T} only if $\Pi_T^a$ is an elliptic projector associated with $a_{|T}$, that is
  \[
  a_{|T}(w,\Pi_T^a\virtual{v}_T)=a_{|T}(w,\virtual{v}_T)\qquad\forall (w,\virtual{v}_T)\in\mathcal P_T\times \virtual{V}_T .
  \]
  In certain cases (e.g., when locally varying physical coefficients are considered), this elliptic projector may not be computable in the sense made precise above.
  Whenever the involved form can be written as
  \[
  a(w,v) = m(\mathcal{D} w, \mathcal{D} v) \qquad \forall (w,v) \in V \times V
  \]
  with $\mathcal{D}$ denoting a differential operator and $m$ a bilinear form, one can often circumvent this problem by letting
  \begin{equation}\label{eq:atilde:modetwo}
    \tilde{\virtual{a}}_T(\virtual{w}_T, \virtual{v}_T)
    \coloneq m_{|T}(\Pi_T^\circ \mathcal{D} \virtual{v}_T, \Pi_T^\circ \mathcal{D} \virtual{w}_T)
    \qquad \forall (\virtual{v}_T, \virtual{w}_T) \in \virtual{V}_T \times \virtual{V}_T,
  \end{equation}
  with $\Pi_T^\circ$ denoting a projector on a polynomial subspace selected in such a way that $\Pi_T^\circ \mathcal{D} \virtual{v}_T$ is computable for any $\virtual{v}_T \in \virtual{V}_T$.
\end{remark}

Let us now discuss the stability of the virtual bilinear form. Condition \eqref{eq:C1} expresses two requirements on the bilinear form $\virtual{s}_T$: that it yields stability and that it scales in $h_T$ like the consistent term.
In the VEM literature, designing a computable stabilization that ensures \eqref{eq:C1} is often one of the most challenging aspects. Available approaches typically follow, implicitly or explicitly, the path indicated by Proposition \ref{prop:C1.C2.E}. First, the virtual stabilization form is expressed as
\begin{equation}\label{eq:design.virtual.sT}
\virtual{s}_T(\virtual{w}_T,\virtual{v}_T)=\discrete{s}_T(\dofmap_T\virtual{w}_T,\dofmap_T\virtual{v}_T)\qquad
  \forall (\virtual{w}_T,\virtual{v}_T)\in\virtual{V}_T\times \virtual{V}_T,
\end{equation}
where $\discrete{s}_T:\discrete{V}_T\times \discrete{V}_T\to\Real$ is a discrete stabilization form designed to ensure \eqref{eq:C2} (with the link \eqref{eq:link:vem.fd} between virtual and discrete bilinear form); this design is often simplified by the choice of the norm on $\discrete{V}_T$. Then, \eqref{eq:C1} is obtained in a second stage, which involves proving (explicitly or implicitly) the norm equivalence \eqref{eq:E}. In Section \ref{sec:nodal.virtual.analysis} we will show, in particular, how a proof of \eqref{eq:E} is actually hidden in classical VEM arguments.
Such norm equivalence (and, in particular, the second inequality expressing the control of the norm of the virtual function by that of its DOFs) actually represents the main difficulty in the design of VEM stabilizations.
A historical survey of stabilization in VEM can be found in \cite{Mascotto:23}.


\section{Application to a nodal scheme for the Poisson equation}\label{sec:application}

We illustrate the abstract framework developed in the previous section by applying it to the convergence analysis of a scheme for the two-dimensional Poisson equation; see Remark \ref{rem:consistency.error}.
We therefore consider a polygonal domain $\Omega \subset \Real^2$, a source term $f\in L^2(\Omega)$, and set
\begin{equation}\label{eq:poisson.2}
  \begin{gathered}
    V = H_0^1(\Omega)\text{ with norm }\norm{V}{v}=\norm{L^2(\Omega;\Real^2)}{\nabla v},\\
    a(w,v) = \int_\Omega \nabla u \cdot \nabla v,\qquad
    \ell(v) = \int_\Omega f v.
  \end{gathered}
\end{equation}
According to \eqref{eq:rel.ell.L}, this sets $L=L^2(\Omega)$.
The scheme we will consider for this problem is based on the $H^1$-like space and operators of the DDR method \cite{Di-Pietro.Droniou:23}. To better connect the virtual approach to this method, we do not follow exactly the original $H^1$-conforming VEM spaces \cite{Beirao-da-Veiga.Brezzi.ea:13, Ahmad.Alsaedi.ea:13,Beirao-da-Veiga.Brezzi.ea:16} but use the variant introduced in \cite[Section 6.2.1]{Beirao-da-Veiga.Dassi.ea:22}.

\subsection{Discrete problem}

\subsubsection{Mesh and polynomial spaces}\label{sec:mesh}

We consider a mesh $(\Th,\Eh,\Vh)$ of $\Omega$, where $\Th$ is the set of disjoint open polygonal elements, $\Eh$ the set of disjoint open edges, and $\Vh$ the set of vertices (endpoints of edges). For each $T\in\Th$ we assume the existence of $\ET\subset\Eh$ such that $\partial T=\cup_{E\in\ET} \overline{E}$, and we assume that each $E\in\Eh$ belongs to at least one $\ET$. The set of all vertices on $\partial T$ is denoted by $\VT$.
We additionally denote by $h_T$ the diameter of $T$ and set $h \coloneqq \max_{T\in\Th} h_T$.

The mesh is assumed to be regular in the following sense: there exists $\varrho>0$ independent of $h$ such that each $T \in \Th$ is star-shaped with respect to a ball of center $x_T\in T$ and radius $\varrho h_T$ and, for all $T\in\Th$ and $E\in\ET$, the diameter of $E$ is $\ge \varrho h_T$.
Under these assumptions, there exists a conforming triangular submesh $\Sh$ which is also regular, and that we will use in Section \ref{sec:fd.analysis.example} below to construct a conforming lifting matching Assumption \ref{ass:conforming.lifting}.
According to the notation $\lesssim$ introduced in Section \ref{sec:abstract.framework:vem.error.analysis}, an inequality of the type $a\le Cb$ with $C$ depending only on $\Omega$ and $\varrho$ will be shortened into $a\lesssim b$.

For $X \in \Th \cup \Eh$ and $\ell \in \Natural$, the space of polynomials on $X$ of total degree $\le \ell$ is denoted by $\Poly{\ell}(X)$, with the additional convention that $\Poly{-1}(X)=\{0\}$. The space $\Poly{\ell}(X;\Real^2)$ gathers all functions $X\to\Real^2$ whose components belong to $\Poly{\ell}(X)$.
The $L^2$-orthogonal projector onto $\Poly{\ell}(X)$ or $\Poly{\ell}(X;\Real^2)$ is denoted by $\lproj{X}{\ell}$.

\subsubsection{Spaces and DOFs}

Let an integer $k \ge 0$ be fixed.
For $T\in\Th$, we pick a complement $\Poly{k-1|k+1}(T)$ of $\Poly{k-1}(T)$ in $\Poly{k+1}(T)$ and
define the local space
\begin{equation}\label{eq:application:def.VT}
  V_T \coloneqq
  \begin{aligned}[t]
    \Big\{
    \virtual{v} \in H^1(T) \;:\;
    &\text{%
      $\virtual{v}_{|E} \in \Poly{k+1}(E)$ for all $E \in \ET$,
      $\Delta \virtual{v} \in \Poly{k+1}(T)$,
    }
    \\
    &\text{
      $\int_T (\nabla \virtual{v} - \lproj{T}{k} \nabla \virtual{v}) \cdot (x - x_T) w = 0$
      for all $w \in \Poly{k-1|k+1}(T)$
    }
    \Big\}.
  \end{aligned}
\end{equation}
At the global level, we set
\[
\widetilde{V}_h \coloneqq \left\{
\virtual{v}_h \in H^1(\Omega) \;:\;
\text{
  $\virtual{v}_T \coloneqq (\virtual{v}_h)_{|T} \in V_T$ for all $T \in \Th$
}
\right\},\qquad
V_h \coloneqq \widetilde{V}_h \cap V.
\]
As a domain for the interpolator we can take, e.g., $V_I = V \cap H^{1+\epsilon}(\Omega)$ with $\epsilon > 0$.
The DOF map $\dofmap_h$ collects:
\begin{itemize}
\item For each element $T \in \Th$, the polynomial moments $\int_T v \, w$ for all $w \in \Poly{k-1}(T)$;

\item For each edge $E \in \Eh$, the polynomial moments $\int_E v \, w$ for all $w \in \Poly{k-1}(E)$;

\item For each vertex $V \in \Vh$, the value $v(x_V)$.
\end{itemize}
For any $X \in \Th \cup \Eh$, the above moments define a unique polynomial function $v_X \in \Poly{k-1}(X)$, which is nothing but the $L^2$-orthogonal projection $\lproj{X}{k-1} v$ of $v$ onto $\Poly{k-1}(X)$.
In view of this fact, identifying polynomial moments and $L^2$-orthogonal projections, we will consider, with a common slight abuse in terminology, that the space of DOFs corresponding to $\widetilde{V}_h$ is
\begin{multline*}
  \widetilde{\discrete{V}}_h \coloneqq \big\{
  \discrete{v}_h = ( (v_T)_{T \in \Th}, (v_E)_{E \in \Eh}, (v_V)_{V \in \Vh} ) \;:\;
  \\
  \text{$v_T \in \Poly{k-1}(T)$ for all $T \in \Th$,
    $v_E \in \Poly{k-1}(E)$ for all $E \in \Eh$,
    $v_V \in \Real$ for all $V \in \Vh$}
  \big\},
\end{multline*}
while $\discrete{V}_h$ is the subspace of $\widetilde{\discrete{V}}_h$ where all edge and vertex boundary components are zero.
Consistently with the above interpretation of DOFs, for any $v \in V_I$, we set
\begin{equation}\label{eq:application:sigmah}
  \dofmap_h v \coloneqq (
  (\lproj{T}{k-1} v)_{T \in \Th},
  (\lproj{E}{k-1} v)_{E \in \Eh},
  (v(x_V))_{V \in \Vh}
  ).
\end{equation}

The VEM discretization of \eqref{eq:poisson.2} is then based on the choices of polynomial spaces
\[
\text{$\mathcal P_T=\Poly{k+1}(T)$ and $\widetilde{\mathcal P}_T=\Poly{k}(T)$ for all $T\in\Th$}.
\]
This choice sets $\Pi^L_T =\lproj{T}{k}$ and we notice that, restricted to $\virtual{V}_T$, this projection is computable from the DOFs (this is a consequence of Lemma \ref{lem:GT.RT.virtual} below). Consequently, $\RTL=\Pi^L_T\circ\dofmap_T^{-1}:\discrete{V}_T\to\Poly{k}(T)$ is well-defined. To describe the bilinear form, we need to first introduce a discrete gradient and potential reconstructions.

\subsubsection{Gradient and potential reconstructions}

Let $\discrete{v}_h = ((v_T)_{T \in \Th}, (v_E)_{E \in \Eh}, (v_V)_{V \in \Vh}) \in \discrete{V}_h$ and, for each $E \in \Eh$, denote by $\discrete{v}_E = ( v_E, (v_V)_{V \in \VE})$ the restriction of $\discrete{v}_h$ to $E$ (with $\VE \subset \Vh$ the set of endpoints of $E$).
We can reconstruct a unique polynomial $\trrec{E} \discrete{v}_E \in \Poly{k+1}(E)$ on $E$ enforcing $\lproj{E}{k-1} \trrec{E} \discrete{v}_E = v_E$ and $\trrec{E} \discrete{v}_E (x_V) = v_V$ for all $V \in \VE$.
For any $T \in \Th$, we let $\trrec{\partial T} \discrete{v}_T\in C(\partial T)$ be such that $(\trrec{\partial T} \discrete{v}_T)_{|E} = \trrec{E} \discrete{v}_E$ for all $E \in \ET$.
We next define the discrete gradient $G_T \discrete{v}_T \in \Poly{k}(T;\Real^2)$ such that
\begin{equation}\label{eq:def.GT}
  \int_T G_T \discrete{v}_T \cdot \tau
  = - \int_T v_T \operatorname{div} \tau
  + \int_{\partial T} \trrec{\partial T} \discrete{v}_T \, (\tau \cdot n_T)
  \qquad \forall \tau \in \Poly{k}(T;\Real^2),
\end{equation}
where $n_T$ denotes the unit normal field on $\partial T$ pointing out of $T$.
Finally, we let $\RTa \discrete{v}_T \in \Poly{k+1}(T)$ be the potential reconstruction such that
\begin{equation}\label{eq:application:RTa}
  \int_T \RTa \discrete{v}_T \operatorname{div} \tau
  = - \int_T G_T \discrete{v}_T \cdot \tau
  + \int_{\partial T} \trrec{\partial T} \discrete{v}_T \, (\tau \cdot n_T)
  \qquad \forall \tau \in (x - x_T) \Poly{k+1}(T)
\end{equation}
(which is well-defined because $\operatorname{div}:(x - x_T) \Poly{k+1}(T)\to \Poly{k+1}(T)$ is an isomorphism) and set
\begin{equation}\label{eq:application:PiTa}
  \Pi_T^a \coloneqq \RTa \circ \dofmap_T.
\end{equation}
It can easily be checked that
\begin{equation}\label{eq:application:GT.RaT:polynomial.consistency}
  \text{
    $\trrec{\partial T}(\dofmap_T w)=w_{|\partial T}$,
    $G_T (\dofmap_T w) = \nabla w$,
    and $\RTa (\dofmap_T w) = \Pi^a_T w = w$
    for all $w \in \Poly{k+1}(T)$.
  }
\end{equation}

\begin{remark}[Validity of \eqref{eq:application:RTa}]\label{eq:application:RTa:validity}
  Using \eqref{eq:def.GT} and the direct decomposition
  \[
  \Poly{k}(T;\Real^2) = {\operatorname{\mathbf{rot}}\Poly{k+1}(T)\oplus (x-x_T)\Poly{k-1}(T)}
  \subset {\operatorname{\mathbf{rot}}\Poly{k+1}(T)\oplus (x-x_T)\Poly{k+1}(T)},
  \]
  and the identity $\operatorname{div} \operatorname{\mathbf{rot}} = 0$, it can be checked that formula \eqref{eq:application:RTa} also holds for any $\tau \in \Poly{k}(T;\Real^2)$.
\end{remark}

\begin{lemma}[Approximation properties of $G_T$]\label{lem:approx.GT}
  For all $T \in \Th$, all $v \in H^{k+2}(T)$, and all $m \in \{0,\ldots,k+1\}$, it holds
  \begin{equation}\label{eq:application:GT:approximation}
    \seminorm{H^m(T;\Real^2)}{\nabla v - G_T \dofmap_T v}
    \lesssim h_T^{k+1-m} \seminorm{H^{k+2}(T)}{v}.
  \end{equation}
\end{lemma}

\begin{proof}
  This results corresponds to Point 2.~in \cite[Proposition 7]{Di-Pietro.Droniou.ea:25} for $n = 2$ applied to $0$-forms.
\end{proof}
 
\begin{lemma}[Approximation properties of $\Pi_T^a$]\label{lem:approx.PiTa}
  For $T \in \Th$,
  all $s \in \{2, \ldots, k+2\}$,
  all $m \in \{0, \ldots, s\}$,
  and all $v \in H^{k+2}(T)$,
  \begin{equation}\label{eq:application:PiTa:approximation}
    \seminorm{H^m(T)}{v - \Pi_T^a v}
    \lesssim h_T^{s-m} \seminorm{H^s(T)}{v}.
  \end{equation}
\end{lemma}

\begin{proof}
  As already noticed, $\Pi_T^a$ is a projector.
  A straightforward adaptation of the arguments of \cite[Proposition 13]{Di-Pietro.Droniou.ea:24} to the enhanced potential reconstruction of \cite[Remark 3.7]{Bonaldi.Di-Pietro.ea:25} additionally shows that, for any $w \in H^2(T)$, $\seminorm{H^{\min(m,2)}(T)}{\Pi_T^a w} \lesssim \sum_{r = \min(m,2)}^2 h_T^r \seminorm{H^r(T)}{w}$.
  The approximation properties \eqref{eq:application:PiTa:approximation} then follow from an application of \cite[Lemma 1.43]{Di-Pietro.Droniou:20}.
\end{proof}

The following lemma highlights the strong links between the reconstruction operators built from the DOFs (in line with the fully discrete approach) and projection operators on the virtual space.

\begin{lemma}[Interpretation of gradient and potential reconstructions in virtual space]\label{lem:GT.RT.virtual}
  For all $T\in\Th$ and all $\virtual{v}_T\in\virtual{V}_T$, we have
  \begin{equation}\label{eq:GT.RT.virtual}
    \text{%
      $\trrec{\partial T}(\dofmap_T\virtual{v}_T)=(\virtual{v}_T)_{|\partial T}$,
      $G_T(\dofmap_T\virtual{v}_T)=\lproj{T}{k}(\nabla \virtual{v}_T)$,
      and $\Pi_T^a\virtual{v}_T=\RTa(\dofmap_T\virtual{v}_T)=\lproj{T}{k+1}\virtual{v}_T$.
    }
  \end{equation}
  As a consequence,
  \begin{equation}\label{eq:RTL=pikT.RaT}
    \RTL=\lproj{T}{k}\circ \RTa.
  \end{equation}
\end{lemma}

\begin{proof}
  By definition \eqref{eq:application:sigmah} of $\dofmap_h$, for all $E\in\ET$, $\trrec{E}(\dofmap_E\virtual{v}_T)$ is the unique polynomial function in $\Poly{k+1}(E)$ that has values $(\virtual{v}_T(x_V))_{V\in\VE}$ at the vertices of $E$ and whose projection on $\Poly{k-1}(E)$ is equal to $\pi^{k-1}_E\virtual{v}_T$; since $(\virtual{v}_T)_{|E}$ is also a polynomial function satisfying these properties, we infer that $\trrec{E}(\dofmap_E\virtual{v}_T)=(\virtual{v}_T)_{|E}$. This proves that $\trrec{\partial T}(\dofmap_T\virtual{v}_T)=(\virtual{v}_T)_{|\partial T}$.

  Using this fact in the definition \eqref{eq:def.GT} of $G_T$, we see that, for all $\tau\in\Poly{k}(T;\Real^2)$,
  \[
  \int_T G_T(\dofmap_T\virtual{v}_T) \cdot \tau = -\int_T (\cancel{\pi^{k-1}_T}\virtual{v}_T)\operatorname{div} \tau + \int_{\partial T}\virtual{v}_T(\tau\cdot n_T)\overset{\text{IBP}}=\int_T \nabla\virtual{v}_T\cdot\tau,
  \]
  where the cancellation is justified noticing that $\operatorname{div} \tau\in\Poly{k-1}(T)$. This concludes the proof of $G_T(\dofmap_T\virtual{v}_T)=\lproj{T}{k}(\nabla \virtual{v}_T)$.

  Let now $\tau\in (x-x_T)\Poly{k+1}(T)$, recall that $\Pi_T^a=\RTa\circ \dofmap_T$, and invoke the definition \eqref{eq:application:RTa} of $\RTa$ with $\discrete{v}_T=\dofmap_T\virtual{v}_T$ together with \eqref{eq:application:PiTa} to get
  \begin{equation}\label{eq:add:X}
    \int_T \Pi_T^a \virtual{v}_T \operatorname{div} \tau = \int_T \RTa(\dofmap_T \virtual{v}_T) \operatorname{div} \tau= - \int_T \lproj{T}{k}(\nabla \virtual{v}_T) \cdot \tau
    + \int_{\partial T} \virtual{v}_T \, (\tau \cdot n_T).
  \end{equation}
  We can write $\tau=(x-x_T)(z+w)$ with $z\in\Poly{k-1}(T)$ and $w\in\Poly{k-1|k+1}(T)$ and thus, by the integral condition in the definition \eqref{eq:application:def.VT} of $\virtual{V}_T$, we continue with
  \begin{align*}
    \int_T \Pi_T^a \virtual{v}_T \operatorname{div} \tau &= - \int_T \cancel{\lproj{T}{k}}(\nabla \virtual{v}_T) \cdot (x-x_T)z
    - \int_T (\nabla \virtual{v}_T) \cdot (x-x_T)w + \int_{\partial T} \virtual{v}_T \, (\tau \cdot n_T)\\
    &= - \int_T (\nabla \virtual{v}_T) \cdot \tau + \int_{\partial T} \virtual{v}_T \, (\tau \cdot n_T)
    \overset{\text{IBP}}=\int_T \virtual{v}_T\operatorname{div}\tau,
  \end{align*}
  the cancellation being justified by $(x-x_T)z\in\Poly{k}(T;\Real^2)$.
  Since $\operatorname{div}\tau$ spans $\Poly{k+1}(T)$ when $\tau$ spans $(x-x_T)\Poly{k+1}(T)$, this concludes the proof that $\Pi_T^a \virtual{v}_T =\RTa(\dofmap_T\virtual{v}_T)=\lproj{T}{k+1}\virtual{v}_T$.

  Since $\lproj{T}{k}\circ\lproj{T}{k+1}=\lproj{T}{k}$ and $\Pi_T^L=\lproj{T}{k}$, we infer that $\Pi_T^L\virtual{v}_T=\lproj{T}{k}(\Pi_T^a\virtual{v}_T)$. Applying this to $\virtual{v}_T=\dofmap_T^{-1}\discrete{v}_T$ for a generic $\discrete{v}_T\in\discrete{V}_T$ and recalling that, for $\bullet\in\{a,L\}$, we have $R^\bullet_T=\Pi^\bullet_T\circ\dofmap_T^{-1}$ proves $\RTL=\lproj{T}{k}\circ \RTa$.
\end{proof}

The following proposition estimates the norms of the computable projections in terms of the DOFs.
  \begin{proposition}[Boundedness of the projections]\label{prop:bound.proj}
    It holds, for all $T \in \Th$ and all $\virtual{v}_T \in \virtual{V}_T$,
    \begin{multline}\label{eq:cont-dof.proj}
      \norm{L^2(T)}{\lproj{T}{k+1} \virtual{v}_T}
      + h_T \norm{L^2(T;\Real^2)}{\lproj{T}{k}(\nabla \virtual{v}_T)}
      \\
      \lesssim
      \norm{L^2(T)}{\lproj{T}{k-1}\virtual{v}_T}
      + h_T^{\frac12} \sum_{E \in \ET} \norm{L^2(E)}{\lproj{E}{k-1}\virtual{v}_T}
      + h_T \sum_{V \in \VT} |\virtual{v}_T(x_V)|.
    \end{multline}
  \end{proposition}

\begin{proof}
  The definition of $L^2$-orthogonal projector and an integration by parts (or, equivalently, exploiting the relation $G_T(\dofmap_T\virtual{v}_T)=\lproj{k}{T}(\nabla\virtual{v}_T)$) yields
  \[
  \begin{aligned}
    \norm{L^2(T)}{\lproj{T}{k}(\nabla \virtual{v}_T)}^2
    & = \int_T \nabla \virtual{v}_T \cdot \lproj{T}{k}(\nabla \virtual{v}_T)
    \overset{\text{IBP}}= - \int_T \virtual{v}_T \operatorname{div}(\lproj{T}{k}(\nabla \virtual{v}_T)) + \int_{\partial T} \virtual{v}_T \, \lproj{T}{k}(\nabla \virtual{v}_T) \cdot n_T \\
    & = - \int_T \lproj{T}{k-1} \virtual{v}_T \, \operatorname{div}(\lproj{T}{k}(\nabla \virtual{v}_T)) + \int_{\partial T} \virtual{v}_T \, \lproj{T}{k}(\nabla \virtual{v}_T) \cdot n_T,
  \end{aligned}
  \]
  where the introduction of $\lproj{T}{k-1}$ is justified since $ \operatorname{div}(\lproj{T}{k}(\nabla \virtual{v}_T))\in\Poly{k-1}(T)$.
  Using Cauchy--Schwarz and standard inverse and trace inequalities for polynomial functions together with
  \begin{equation}\label{eq:estim.poly.edge}
    \norm{L^2(E)}{v}^2 \simeq \norm{L^2(E)}{\lproj{E}{k-1} v}^2 + h_T \sum_{V \in \VE} |v(x_V)|^2\qquad\forall E\in\Eh\,,\quad \forall v \in \Poly{k+1}(E),
  \end{equation}
  we infer that
  \begin{equation}\label{eq:cont-dof.grad}
    \norm{L^2(T;\Real^2)}{\lproj{T}{k}(\nabla \virtual{v}_T)}^2
    \lesssim
    h_T^{-2} \norm{L^2(T)}{\lproj{T}{k-1}\virtual{v}_T}^2
    + h_T^{-1} \sum_{E \in \ET} \norm{L^2(E)}{\lproj{E}{k-1}\virtual{v}_T}^2
    +  \sum_{V \in \VT} |\virtual{v}_T(x_V)|^2,
  \end{equation}
  from which we can easily deduce the sought estimate for the second term in the left-hand side of \eqref{eq:cont-dof.proj}.
  \smallskip
  
  To estimate the first term, we start by introducing
  $\tau \in (x-x_T)\Poly{k+1}(T)$ such that $\operatorname{div}\tau = \lproj{T}{k+1} \virtual{v}_T$ and satisfying
  $
  \norm{L^2(T;\Real^2)}{\tau}
  + h_T^{\frac12} \norm{L^2(\partial T;\Real^2)}{\tau} \lesssim h_T \norm{L^2(T)}{\lproj{T}{k+1} \virtual{v}_T} .
  $
  The existence of such $\tau$ follows from the surjectivity of the involved operator and polynomial scaling arguments \cite[Lemma 9]{Di-Pietro.Droniou:23}.
  Recalling the relations in Lemma \ref{lem:GT.RT.virtual}, we can apply \eqref{eq:add:X} for such $\tau$, yielding
  \[
  \norm{L^2(T)}{\lproj{T}{k+1} \virtual{v}_T}^2 = \int_T (\lproj{T}{k+1} \virtual{v}_T) \, \operatorname{div} \tau
  = - \int_T \lproj{T}{k}(\nabla \virtual{v}_T) \cdot \tau
  + \int_{\partial T} \virtual{v}_T \, (\tau \cdot n_T).
  \]
  We conclude using a Cauchy--Schwarz inequality, \eqref{eq:cont-dof.grad}, \eqref{eq:estim.poly.edge}, and the above bounds for $\tau$.
\end{proof}

\subsubsection{Discrete forms}

The bilinear form $\discrete{a}_h$ is obtained according to \eqref{eq:discrete:ah} letting, for all $T \in \Th$, $\discrete{a}_T : \discrete{V}_T \times \discrete{V}_T \to \Real$ be such that, for all $(\discrete{w}_T, \discrete{v}_T) \in \discrete{V}_T \times \discrete{V}_T$,
\begin{equation}\label{eq:application:aT}
  \discrete{a}_T(\discrete{w}_T, \discrete{v}_T)
  \coloneqq \int_T G_T \discrete{w}_T \cdot G_T \discrete{v}_T +
  \discrete{s}_T(\discrete{w}_T - \dofmap_T \RTa \discrete{w}_T, \discrete{v}_T - \dofmap_T \RTa \discrete{v}_T),
\end{equation}
with $\discrete{s}_T : \discrete{V}_T \times \discrete{V}_T \to \Real$ such that
\begin{equation}\label{eq:application:sT}
  \discrete{s}_T(\discrete{w}_T, \discrete{v}_T)
  = h_T^{-1} \sum_{E \in \ET} \int_E w_E \, v_E
  + \sum_{V \in \VT} w_V \, v_V .
\end{equation}
In the virtual setting, following the approach \eqref{eq:atilde:modetwo}, the forms above can be equivalently written as (see Lemma \ref{lem:GT.RT.virtual})
\begin{equation}\label{eq:application:aT:cont}
  \virtual{a}_T(\virtual{w}_T, \virtual{v}_T)
  \coloneqq \int_T \lproj{T}{k} (\nabla \virtual{w}_T) \cdot \lproj{T}{k} (\nabla \virtual{v}_T) +
  \virtual{s}_T(\virtual{w}_T - \lproj{T}{k+1} \virtual{w}_T, \virtual{v}_T - \lproj{T}{k+1} \virtual{v}_T),
\end{equation}
with $\virtual{s}_T : \virtual{V}_T \times \virtual{V}_T \to \Real$ such that
\begin{equation}\label{eq:application:sT:cont}
  \virtual{s}_T(\virtual{w}_T, \virtual{v}_T)
  =  h_T^{-1} \sum_{E \in \ET} \int_E (\lproj{E}{k-1} \virtual{w}_T) \, (\lproj{E}{k-1} \virtual{v}_T)
  + \sum_{V \in \VT} \virtual{w}_T(x_V) \, \virtual{v}_T(x_V) .
\end{equation}
Since we have chosen $\mathcal{P}_T=\Poly{k+1}(T)$, the bilinear form \eqref{eq:application:aT:cont} obviously satisfies the consistency property \eqref{eq:virtual:consistency}.
Note that there are many other possible choices in the VEM literature for the stabilization form, depending, for instance, on the considered mesh assumptions or other considerations of more practical nature \cite{Beirao-da-Veiga.Brezzi.ea:14,Beirao-da-Veiga.Lovadina.ea:17,Brenner.Sung:18,Wriggers.Aldakheel.ea:24}.

\begin{remark}[Equivalent stabilization bilinear form]\label{rem:stab:equiv}
  Exploiting the fact that the restrictions of the virtual functions to the mesh edges are polynomials of degree $\le k + 1$ and recalling \eqref{eq:estim.poly.edge},
  it is trivial to check that the form $\virtual{s}_T$ above can be equivalently substituted by
  \[
  \tilde{\virtual{s}}_T(\virtual{w}_T,\virtual{v}_T) = h_T^{-1} \int_{\partial T} \virtual{w}_T \, \virtual{v}_T \, ,
  \]
  which is classical in the VEM literature.
  Indeed, the two forms are equivalent in the sense that
  \begin{equation}\label{eq:application:sT:cont-bis}
    \virtual{s}_T(\virtual{w}_T, \virtual{w}_T)
    \simeq \tilde{\virtual{s}}_T(\virtual{w}_T,\virtual{w}_T)
    =  h_T^{-1} \norm{L^2(\partial T)}{\virtual{w}_T}^2 \qquad \forall \virtual{w}_T \in \virtual{V}_T.
  \end{equation}
\end{remark}

\begin{proposition}[Consistency of $\discrete{s}_T$]
  For all $T \in \Th$ and all $v \in H^{k+2}(T)$, it holds
  \begin{equation}\label{eq:application:sT:consistency}
    \discrete{s}_T(\dofmap_T (v - \Pi_T^a v), \dofmap_T (v - \Pi_T^a v))^{\frac12}
    \lesssim h_T^{k+1} \seminorm{H^{k+2}(T)}{v}.
  \end{equation}
\end{proposition}

\begin{proof}
  We have
  \begin{equation*} 
    \begin{aligned}
      \discrete{s}_T(\dofmap_T {}&(v - \Pi_T^a v), \dofmap_T (v - \Pi_T^a v))
      \\
      &\begin{aligned}[t]
         \overset{\eqref{eq:application:sT},\eqref{eq:application:sigmah}}&=
         h_T^{-1} \sum_{E \in \ET} \norm{L^2(E)}{\lproj{E}{k-1} (v - \Pi_T^a v)}^2
         + \sum_{V \in \VT} |(v - \Pi_T^a v)(x_V)|^2
         \\
         &\le
         h_T^{-1} \sum_{E \in \ET} \norm{L^2(E)}{v - \Pi_T^a v}^2
         + \sum_{V \in \VT} |(v - \Pi_T^a v)(x_V)|^2
         \\
         &\lesssim
         h_T^{-2} \norm{L^2(T)}{v - \Pi_T^a v}^2
         + \seminorm{H^1(T)}{v - \Pi_T^a v}^2
         + h_T^2 \seminorm{H^2(T)}{v - \Pi_T^a v}^2,
       \end{aligned}
    \end{aligned}
  \end{equation*}
  where we have used the $L^2(E)$-boundedness of $\lproj{E}{k-1}$ in the second step and continuous trace inequalities in the third step.
  Use \eqref{eq:application:PiTa:approximation} with $s = k + 2$ and $m \in \{0, 1, 2\}$ to estimate the terms in the right-hand side and conclude.
\end{proof}

The loading term is discretized by the linear form $\ell_h:\discrete{V}_h\to\Real$ such that
\begin{equation}\label{eq:application:ell.h:RTL:discrete}
  \ell_h(\discrete{v}_h) = \sum_{T \in \Th} \int_T f \, \RTL \discrete{v}_T
  \qquad \forall \discrete{v}_h \in \discrete{V}_h.
\end{equation}
Its virtual counterpart, still denoted by $\ell_h$ with a little abuse of notation, is $\ell_h : \virtual{V}_h \to \Real$ such that
\begin{equation}\label{eq:application:ell.h:RTL:virtual}
  \ell_h(\virtual{v}_h) = \sum_{T \in \Th} \int_T f \, \lproj{T}{k}\virtual{v}_T
  \qquad \forall \virtual{v}_h \in \virtual{V}_h.
\end{equation}

\subsection{Fully discrete norm and discrete coercivity}

The fully discrete space $\discrete{V}_h$ is equipped with the norm such that, for all $\discrete{v}_h \in \discrete{V}_h$,
\begin{equation}\label{eq:application:discrete.norm}
  \begin{gathered}
    \norm{\discrete{V}_h}{\discrete{v}_h} = \left(
    \sum_{T \in \Th} \norm{\discrete{V}_T}{\discrete{v}_T}^2
    \right)^{\frac12},
    \\
    \norm{\discrete{V}_T}{\discrete{v}_T}^2
    \coloneqq
    \norm{L^2(T;\Real^2)}{G_T \discrete{v}_T}^2
    + h_T^{-1}\norm{L^2(\partial T)}{\RTa \discrete{v}_T - \trrec{\partial T} \discrete{v}_T}^2.
  \end{gathered}
\end{equation}
The following boundedness property for the potential reconstruction can be proved taking $\tau = \nabla \RTa \discrete{v}_T$ in \eqref{eq:application:RTa} (this is possible in view of Remark \ref{eq:application:RTa:validity}), integrating by parts the left-hand side, and using Cauchy--Schwarz and discrete trace inequalities to estimate the right-hand side:
\begin{equation}\label{eq:application:RTa:boundedness}
  \norm{L^2(T;\Real^2)}{\nabla \RTa \discrete{v}_T}
  \lesssim \norm{\discrete{V}_T}{\discrete{v}_T}
  \qquad \forall \discrete{v}_T \in \discrete{V}_T.
\end{equation}
Moreover, using \eqref{eq:estim.poly.edge} together with $\sigma_{\partial T}(\gamma_{\partial T}\discrete{v}_T)=\discrete{v}_{\partial T}$, it is a simple matter to prove that
\begin{equation}\label{eq:application:ah:coercivity.boundedness}
  \discrete{a}_T(\discrete{v}_T, \discrete{v}_T) \simeq \norm{\discrete{V}_T}{\discrete{v}_T}^2\qquad
  \forall T\in\Th,\quad \forall \discrete{v}_T\in\discrete{V}_T.
\end{equation}

\begin{remark}[DOF-based $H^1$-like norm]
The discrete norm \eqref{eq:application:discrete.norm} is an $H^1$-like (energy) norm associated to the bilinear form $\discrete{a}_h$ and thus built using the reconstruction operators $G_T$, $\RTa$ and $\gamma_{\partial T}$. One can also construct and $H^1$-like norm directly from the DOFs themselves, by setting: for all $\discrete{v}_h\in\discrete{V}_h$,
  \begin{equation}\label{eq:def.H1.norm}
  \begin{gathered}
      \norm{1,h}{\discrete{v}_h}\coloneq\left(\sum_{T\in\Th}\norm{1,T}{\discrete{v}_T}^2\right)^{\frac12},\\
      \norm{1,T}{\discrete{v}_T}\coloneq \left(\norm{L^2(T)}{\nabla v_T}^2+h_T^{-1}\sum_{E\in\ET}\norm{L^2(E)}{v_T-v_E}^2
      +\sum_{V\in\VT}|v_T(x_V)-v_V|^2\right).
  \end{gathered}
  \end{equation}
  Using discrete trace and inverse inequalities, together with \eqref{eq:estim.poly.edge}, and following similar arguments as in \cite[Proposition 2.13]{Di-Pietro.Droniou:20}, it can be shown that $\norm{1,h}{\cdot}\simeq\norm{\discrete{V}_h}{\cdot}$. 
  DOF-norms like \eqref{eq:def.H1.norm} are typically involved in the proof of the approximation properties of reconstruction operators (such as Lemmas \ref{lem:approx.GT} and \ref{lem:approx.PiTa}).
\end{remark}

\subsection{Convergence analysis based on virtual functions}\label{sec:nodal.virtual.analysis}

\subsubsection{Basic analysis}

In light of Theorem \ref{thm:vem:error.estimate}, the first and main step is showing the validity of the inf-sup condition \eqref{eq:virtual:inf-sup} and the local boundedness condition \eqref{eq:virtual:local-bound}. More precisely, the coercive nature of the present problem allows to show \eqref{eq:virtual:inf-sup} as a trivial consequence of the coercivity property
\begin{equation}\label{eq:virtual:coerc}
  \norm{V}{\virtual{w}_h}^2 \lesssim \virtual{a}_h(\virtual{w}_h,\virtual{w}_h) \qquad \forall \virtual{w}_h \in \virtual{V}_h.
\end{equation}
We therefore focus on \eqref{eq:virtual:coerc} and \eqref{eq:virtual:local-bound}.
\begin{proposition}[Properties of the virtual bilinear form]\label{prop:main-cont}
  The virtual bilinear form $\virtual{a}_h$ satisfies the coercivity property \eqref{eq:virtual:coerc} and local boundedness property \eqref{eq:virtual:local-bound}.
\end{proposition}
\begin{proof}
  \textit{Part 1.} We start by showing that
  \begin{equation}\label{eq:simple}
    \virtual{s}_T(\virtual{v}_T,\virtual{v}_T) \simeq
    a_{|T}(\virtual{v}_T,\virtual{v}_T)
    = \seminorm{H^1(T)}{\virtual{v}_T}^2
    \qquad \forall T \in \Th , \ \forall \virtual{v}_T \in \virtual{V}_T \textrm{ with } \lproj{T}{k+1} \virtual{v}_T = 0.
  \end{equation}
  By definition \eqref{eq:application:sT:cont} of $\virtual{s}_T$, the condition $\virtual{s}_T(\virtual{v}_T,\virtual{v}_T) \lesssim a_{|T}(\virtual{v}_T,\virtual{v}_T)$ in \eqref{eq:simple} follows immediately from the equivalence in \eqref{eq:application:sT:cont-bis}, a scaled trace inequality, and the approximation properties of $\lproj{T}{k+1}$ applied to $\virtual{v}_T=\virtual{v}_T-\lproj{T}{k+1}\virtual{v}_T$:
  \[
  \virtual{s}_T(\virtual{v}_T,\virtual{v}_T)
  \lesssim h_T^{-1} \norm{L^2(\partial T)}{\virtual{v}_T}^2
  \lesssim h_T^{-2} \norm{L^2(T)}{\virtual{v}_T}^2 + \seminorm{H^1(T)}{\virtual{v}_T}^2
  \lesssim \seminorm{H^1(T)}{\virtual{v}_T}^2 .
  \]
  Regarding the condition $a_{|T}(\virtual{v}_T,\virtual{v}_T) \lesssim \virtual{s}_T(\virtual{v}_T,\virtual{v}_T)$ in \eqref{eq:simple}, we apply an integration by parts and use the fact that $\Delta \virtual{v}_T \overset{\eqref{eq:application:def.VT}}\in \Poly{k+1}(T)$ combined with $\lproj{T}{k+1} \virtual{v}_T=0$ to write
  \begin{align}
    \seminorm{H^1(T)}{\virtual{v}_T}^2 &=
    - \int_T  (\Delta \virtual{v}_T) \virtual{v}_T \, + \, \langle \nabla \virtual{v}_T\cdot n_T , \virtual{v}_T \rangle_{\partial T}\nonumber\\
    &= \langle \nabla \virtual{v}_T\cdot n_T , \virtual{v}_T \rangle_{\partial T}
    \lesssim \tnorm{H^{\frac12}(\partial T)}{\virtual{v}_T}
    \, \tnorm{H^{-\frac12}(\partial T)}{\nabla \virtual{v}_T\cdot n_T},
    \label{eq:m1}
  \end{align}
  where $\tnorm{H^{\frac12}(\partial T)}{\cdot} := h_T^{-\frac12} \norm{L^2(\partial T)}{\cdot} + \seminorm{H^{\frac12}(\partial T)}{\cdot}$, $\tnorm{H^{-\frac12}(\partial T)}{\cdot}$ is its dual norm, and $\langle\cdot,\cdot\rangle_{\partial T}$ is the duality pairing between $H^{\frac12}(\partial T)$ and $H^{-\frac12}(\partial T)$.
  By a standard scaling argument, since $\virtual{v}_T$ on $\partial T$ is a piecewise polynomial (continuous) function, again recalling \eqref{eq:application:sT:cont-bis} it follows
  \begin{equation}\label{eq:m2}
    \tnorm{H^{\frac12}(\partial T)}{\virtual{v}_T} \lesssim
    h_T^{-\frac12} \norm{L^2(\partial T)}{\virtual{v}_T}
    = s_T(\virtual{v}_T,\virtual{v}_T)^{\frac12}.
  \end{equation}
  On the other hand, Lemma \ref{lem:scal:trace} applied to $\nabla\virtual{v_T}$ yields
  \[
  \tnorm{H^{-\frac12}(\partial T)}{\nabla \virtual{v}_T\cdot n_T}
  \lesssim
  \seminorm{H^1(T)}{\virtual{v}_T}
  + h_T \seminorm{L^2(T)}{\Delta \virtual{v}_T} .
  \]
  A classical scaling argument, exploiting that $\Delta \virtual{v}_T \in \Poly{k+1}(T)$ (see for instance \cite{Beirao-da-Veiga.Lovadina.ea:17}), gives from the above bound
  \begin{equation}\label{eq:m3}
    \tnorm{H^{-\frac12}(\partial T)}{\nabla \virtual{v}_T\cdot n_T} \lesssim \seminorm{H^1(T)}{\virtual{v}_T} .
  \end{equation}
  Substitution of \eqref{eq:m2} and \eqref{eq:m3} into \eqref{eq:m1} yields
  \begin{equation}\label{eq:bound.vT.H1.sT}
    \seminorm{H^1(T)}{\virtual{v}_T} \lesssim s_T(\virtual{v}_T,\virtual{v}_T)^{\frac12} ,
  \end{equation}
  which concludes the proof of \eqref{eq:simple}.
  \medskip\\
  \textit{Part 2.} We now prove \eqref{eq:virtual:coerc} and \eqref{eq:virtual:local-bound}.
  By the orthogonality properties of $\lproj{T}{k}$, the Pythagorean inequality gives, for any $T\in\Th$ and any $\virtual{w}_T \in \virtual{V}_T$,
  \begin{align}
    \seminorm{H^1(T)}{\virtual{w}_T}^2
    &=  \norm{L^2(T;\Real^2)}{\lproj{T}{k}(\nabla\virtual{w}_T)}^2
    + \norm{L^2(T;\Real^2)}{\nabla \virtual{w}_T - \lproj{T}{k}(\nabla \virtual{w}_T)}^2
    \nonumber\\
    &\le \norm{L^2(T;\Real^2)}{\lproj{T}{k}(\nabla\virtual{w}_T)}^2
    + \norm{L^2(T; \Real^2)}{\nabla (\virtual{w}_T - \lproj{T}{k+1}\virtual{w}_T)}^2
    \nonumber\\
    &=
    \norm{L^2(T; \Real^2)}{\lproj{T}{k}(\nabla\virtual{w}_T)}^2
    + a_{|T}(\virtual{w}_T - \lproj{T}{k+1}\virtual{w}_T, \virtual{w}_T - \lproj{T}{k+1}\virtual{w}_T)
    \nonumber\\
    \overset{\eqref{eq:simple}}&\lesssim \int_T | \lproj{T}{k} (\nabla \virtual{w}_T) |^2
    + s_T(\virtual{w}_T - \lproj{T}{k+1}\virtual{w}_T,\virtual{w}_T - \lproj{T}{k+1}\virtual{w}_T)
    = \virtual{a}_T(\virtual{w}_T,\virtual{w}_T) \, ,
    \label{eq:for.norm.equivalence}
  \end{align}
  where we have used the fact that $\lproj{T}{k} \nabla \virtual{w}_T$ is the element of $\Poly{k}(T; \Real^2)$ that minimises the $L^2$-distance from $\nabla \virtual{w}_T$ in the second step, and is therefore closer to this function than $\nabla (\lproj{T}{k+1}\virtual{w}_T)\in\Poly{k}(T;\Real^2)$. The coercivity condition \eqref{eq:virtual:coerc} follows by summing on all the elements $T \in \Th$.

  The local boundedness condition \eqref{eq:virtual:local-bound} follows immediately for the first term (consistency contribution) in the definition \eqref{eq:application:aT:cont} of $\virtual{a}_T$ due to the $L^2$-boundedness of $\lproj{T}{k}$.
  We therefore focus on showing \eqref{eq:virtual:local-bound} for the stability term in \eqref{eq:application:aT:cont}. Applying \eqref{eq:simple} we can write, for all $T\in\Th$ and all $\virtual{w}_T \in \virtual{V}_T $,
  \begin{equation}\label{eq:for.norm.equivalence.2}
    \virtual{s}_T(\virtual{w}_T - \lproj{T}{k+1}\virtual{w}_T,\virtual{w}_T - \lproj{T}{k+1}\virtual{w}_T )
    \lesssim \seminorm{H^1}{\virtual{w}_T - \lproj{T}{k+1}\virtual{w}_T}^2
    \lesssim \seminorm{H^1}{\virtual{w}_T}^2
    + \seminorm{H^1}{\lproj{T}{k+1}\virtual{w}_T}^2
    \lesssim \seminorm{H^1}{\virtual{w}_T}^2 ,
  \end{equation}
  where we applied standard $H^1$-boundedness results for $L^2$-orthogonal projections on polynomials in the last bound.
  Bound \eqref{eq:virtual:local-bound} now follows from the Cauchy--Schwarz inequality, since the involved forms are semi positive-definite and symmetric.
\end{proof}

We now state interpolation estimates for the virtual spaces.
From this point on, for any integer $m \ge 0$, we let $H^m(\Th) \coloneqq \left\{ v \in L^2(\Omega) \;:\; \text{$v_{|T} \in H^m(T)$ for all $T \in \Th$}\right\}$ denote the broken Sobolev space of degree $m$.

\begin{proposition}[Interpolation estimates for the VEM spaces] \label{prop:interp:vem}
  If $u \in H^{k+2}(\Th)$, for all $T \in \Th$ it holds
  \[
  \seminorm{H^1(T)}{u - I_T u} \lesssim h_T^{k+1} \seminorm{H^{k+2}(T)}{u}.
  \]
\end{proposition}
\begin{proof}
  Let $T \in \Th$ and $\virtual{v}_T \in \virtual{V}_T$.
  Using a triangle inequality, we get
  \begin{align*}
    \seminorm{H^1(T)}{\virtual{v}_T}&\le \seminorm{H^1(T)}{\virtual{v}_T-\lproj{T}{k+1}\virtual{v}_T}+\seminorm{H^1(T)}{\lproj{T}{k+1}\virtual{v}_T}\\
    \overset{\eqref{eq:bound.vT.H1.sT}}&\lesssim s_T(\virtual{v}_T-\lproj{T}{k+1}\virtual{v}_T,\virtual{v}_T-\lproj{T}{k+1}\virtual{v}_T)^{\frac12}+h_T^{-1}\norm{L^2(T)}{\lproj{T}{k+1}\virtual{v}_T}\\
    \overset{\eqref{eq:application:sT:cont-bis}}&\lesssim h_T^{-\frac12}\norm{L^2(\partial T)}{\virtual{v}_T-\lproj{T}{k+1}\virtual{v}_T}+
    h_T^{-1}\norm{L^2(T)}{\lproj{T}{k+1}\virtual{v}_T}\\
    &\lesssim h_T^{-\frac12}\norm{L^2(\partial T)}{\virtual{v}_T}+
    h_T^{-1}\norm{L^2(T)}{\lproj{T}{k+1}\virtual{v}_T},
  \end{align*}
  where we have additionally used an inverse inequality for polynomials in the second step,
  while the conclusion was obtained using a triangle inequality and a discrete trace inequality.
    Invoking, respectively, \eqref{eq:estim.poly.edge} and \eqref{eq:cont-dof.proj} for the terms in the right-hand side of the above inequality, we then get
  \begin{equation}\label{eq:cor-dof-bound}
    \seminorm{H^1(T)}{\virtual{v}_T} \lesssim
    h_T^{-1} \norm{L^2(T)}{\lproj{T}{k-1} \virtual{v}_T} +
    h_T^{-\frac12} \sum_{E \in \ET} \norm{L^2(\partial E)}{\pi_E^{k-1}\virtual{v}_{T}}
    + \sum_{V \in \VT} |\virtual{v}_T(x_V)|.
  \end{equation}
  Notice that the right hand side above depends only on the degree of freedom values of $\virtual{v}_T$.
  Let now $p \in \Poly{k+1}(T)$. First by the triangle inequality and using that $\Poly{k+1}(T) \subseteq V_T$,
  then applying bound \eqref{eq:cor-dof-bound} to the function $I_T(p - u) \in \virtual{V}_T$ and using that
  $\dofmap_T I_T (p - u) \overset{\eqref{eq:virtual:Ih}}= \dofmap_T(p - u)$,
  we derive
  \[
  \begin{aligned}
    \seminorm{H^1(T)}{u - I_T u}
    &\le \seminorm{H^1(T)}{u - p} + \seminorm{H^1(T)}{I_T (p - u)}
    \\
    &\lesssim \seminorm{H^1(T)}{u - p}
    + h_T^{-1} \norm{L^2(T)}{\lproj{T}{k-1}(p-u)}
    \\
    &\quad
    + h_T^{-\frac12} \sum_{E \in \ET} \norm{L^2(\partial E)}{\pi_E^{k-1}(p-u)}
    + \sum_{V \in \VT} |(p-u)(x_V)| .
  \end{aligned}
  \]
  By the boundedness of the involved $L^2$-orthogonal projectors and using polynomial approximation bounds for a suitable $p \in \Poly{k+1}(T)$, we obtain the desired interpolation estimate.
\end{proof}

\begin{remark}[Alternative approach for interpolation estimates on virtual spaces]\label{rem:vem:Ih}
  The above interpolation estimate represents a possible way to prove approximation estimates for the VEM space, based on the fact that the space contains polynomials and using the stability estimate. Another approach often found in the literature (see, e.g.~\cite{Mora.Rivera.ea:15, Beirao-da-Veiga.Vacca:22}) is to build a virtual function such that, in addition to interpolating the objective function $u$ on the edges, has a Laplacian which approximates the Laplacian of $u$. Such approach builds the approximation estimate directly rather than passing through $\Poly{k+1}(T) \subseteq \virtual{V}_T$, but the ensuing virtual function is not interpolatory for all DOFs.
\end{remark}

We can now state the final converge estimate.

\begin{theorem}[Error estimate based on virtual functions]\label{thm:vem:conv:model}
    Assume that the solution to \eqref{eq:weak} satisfies $u\in H^{k+2}(\Th)$, that $f\in H^{k+1}(\Th)$, and that $\virtual{u}_h \in \virtual{V}_h$ is the solution to \eqref{eq:virtual:discrete} with the local bilinear forms and global linear form defined by \eqref{eq:application:aT:cont} and \eqref{eq:application:ell.h:RTL:virtual}, respectively.
  Then, it holds
  \begin{equation}\label{eq:vem:conv}
    \norm{H^1(\Omega)}{u - \virtual{u}_h}
    \lesssim h^{k+1} \left(
    \seminorm{H^{k+2}(\Th)}{u} + \seminorm{H^{k+1}(\Th)}{f}
    \right).
  \end{equation}
\end{theorem}
\begin{proof}
Due to Proposition \ref{prop:main-cont}, we can apply Theorem \ref{thm:vem:error.estimate}.
The first term in the right-hand side of \eqref{eq:vem:error.estimate} is bounded using Proposition \ref{prop:interp:vem}.
The last two terms in the right-hand side of \eqref{eq:vem:error.estimate} are trivial to bound by polynomial approximation properties on star-shaped domains. Note that the regularity requirement on $f$ can be easily weakened, see Remark \ref{rem:load:ort}).
\end{proof}

\subsubsection{Uncovering the norm equivalence}

As discussed in Section \ref{sec:practical.construction.vem}, the design of the virtual stabilization bilinear form $\virtual{s}_T$ is usually done from the DOFs in the spirit of \eqref{eq:design.virtual.sT}.
This is plainly visible from the expressions \eqref{eq:application:sT} and \eqref{eq:application:sT:cont}.
The stability of the resulting virtual bilinear form $\virtual{a}_T$ is then established through arguments strongly related to the norm equivalence \eqref{eq:E}.
Let us exemplify this principle on the method considered here.

For all $\virtual{w}_T\in\virtual{V}_T$, we have
\[
\begin{aligned}
  \norm{\discrete{V}_T}{\dofmap_T\virtual{w}_T}^2\overset{\eqref{eq:application:discrete.norm}}&= \norm{L^2(T;\Real^2)}{G_T(\dofmap_T\virtual{w}_T)}^2+
  h_T^{-1}\norm{L^2(\partial T)}{\RTa\dofmap_T\virtual{w}_T-\trrec{\partial T}(\dofmap_T\virtual{w}_T)}^2\\
  \overset{\eqref{eq:GT.RT.virtual}}&= \norm{L^2(T;\Real^2)}{\lproj{T}{k}(\nabla\virtual{w}_T)}^2+
  h_T^{-1}\norm{L^2(\partial T)}{\lproj{T}{k+1}\virtual{w}_T-\virtual{w}_T}^2\\
  \overset{\eqref{eq:application:sT:cont-bis}}&\simeq\norm{L^2(T;\Real^2)}{\lproj{T}{k}(\nabla\virtual{w}_T)}^2+
  \virtual{s}_T(\lproj{T}{k+1}\virtual{w}_T-\virtual{w}_T,\lproj{T}{k+1}\virtual{w}_T-\virtual{w}_T).
\end{aligned}
\]
The bound \eqref{eq:for.norm.equivalence} then shows that $\seminorm{H^1(T)}{\virtual{w}_T}^2 \lesssim \norm{\discrete{V}_T}{\dofmap_T\virtual{w}_T}^2$, while the estimate \eqref{eq:for.norm.equivalence.2} and the $L^2(T)$-boundedness of $\lproj{T}{k}$ give $\norm{\discrete{V}_T}{\dofmap_T\virtual{w}_T}^2\lesssim\seminorm{H^1(T)}{\virtual{w}_T}^2$.
This shows that the key estimates \eqref{eq:for.norm.equivalence} and \eqref{eq:for.norm.equivalence.2}, established to carry out the usual VEM analysis of the bilinear form, directly translate into the norm equivalence
\[
\norm{\discrete{V}_T}{\dofmap_T\virtual{w}_T}\simeq \seminorm{H^1(T)}{\virtual{w}_T},
\]
which is precisely \eqref{eq:E} in our setting.

\begin{remark}[Use of DOF-based norms in VEM analysis]
Although not always explicitly identified, DOF-based norms are often used in VEM methods, at the design stage (e.g., through the popular ``DOFI-DOFI'' stabilization, which is based on a scaled Euclidean norm of the DOF values) and during their analysis. An illustration of the latter is the $L^2$-like DOF-based norm appearing in right-hand side of \eqref{eq:cont-dof.proj} (note that, with minor modifications in the proof of Proposition \ref{prop:interp:vem}, we could have also used the $H^1$-like DOF norm \eqref{eq:def.H1.norm}).
It should however be noted that alternative forms of analysis of VEM can be carried out without explicitly resorting to DOF norms, but using known polynomial approximation properties.
\end{remark}

\subsection{Convergence analysis based on a conforming lifting}\label{sec:fd.analysis.example}

We construct a lifting satisfying Assumption \ref{ass:conforming.lifting} with $\mathcal P_T \coloneqq \Poly{k+1}(T)$.
The existence of such lifting and \eqref{eq:application:ah:coercivity.boundedness} (which implies \eqref{eq:discrete:inf-sup} and \eqref{eq:discrete.a.T:boundedness}) make it possible to invoke Theorem \ref{thm:discrete:error.estimate} to derive an error estimate for the scheme, as shown in Theorem \ref{thm:error.estimates.via.lifting} below.
\smallskip

Let $\discrete{v}_h\in\discrete{V}_h$ and denote by $\trrec{\skel}\discrete{v}_h$ the piecewise polynomial function on the mesh skeleton $\skel \coloneqq \bigcup_{E\in\Eh} \overline{E}$ such that $(\trrec{\skel}\discrete{v}_h)_{|E}=\trrec{E}\discrete{v}_E\in\Poly{k+1}(E)$ for all $E\in\Eh$.
By construction, $\trrec{\skel}\discrete{v}_h \in C^0(\Gamma_h)$.

The lifting $L_h\discrete{v}_h$ is constructed as
\begin{equation}\label{eq:application:Lh}
  L_h\discrete{v}_h \coloneqq L^1_h\discrete{v}_h + L^2_h\discrete{v}_h,
\end{equation}
where the contribution $L^1_h\discrete{v}_h$ lifts the skeleton function $\trrec{\skel}\discrete{v}_h$ inside the elements, while $L^2_h\discrete{v}_h$ corrects this initial lifting in order to achieve the consistency and projection properties \eqref{eq:Lh:consistency:a} and \eqref{eq:Lh:consistency:ell}.

Recall that, by mesh regularity assumption, there is a regular conforming triangular submesh $\Sh$ of $\Th$.
Let $\Poly{k+1}_{c,0}(\Sh)\subset V$ be the Lagrange space of continuous functions on $\Omega$, that are piecewise polynomial of total degree $\le k+1$ on $\Sh$, and that vanish on $\partial\Omega$. Let $\mathfrak N_h$ be the set of Lagrange nodes in this space and define a lifting $L^1_h\discrete{v}_h\in\Poly{k+1}_{c,0}(\Sh)$ such that, for all $P \in \mathfrak N_h$,
\begin{equation}\label{eq:def.L1h}
  (L^1_h\discrete{v}_h)(x_P) =
  \begin{cases}
    (\trrec{\skel}\discrete{v}_h)(x_P)&\text{ if $P \in \skel$,}\\
    (\RTa\discrete{v}_T)(x_P)&\text{otherwise, with $T$ such that $p\in T$.}
  \end{cases}
\end{equation}
This choice uniquely defines the nodal values of $L^1_h\discrete{v}_h$ since $\trrec{\skel}\discrete{v}_h$ is continuous on $\skel$ and since, if $P \not \in \skel$, there exists a unique $T \in \Th$ such that $p\in T$.

The correction $L^2_h\discrete{v}_h$ is made of bubble functions inside each element. To do so, we start by selecting, for each $T\in\Th$, a function $b_T$ such that
\begin{equation}\label{eq:def.bT}
  \begin{aligned}
    &\text{$b_T:T\to [0,1]$ is Lipschitz-continuous, $(b_T)_{|\partial T}=0$,}\\
    &\text{$b_T\gtrsim 1$ on a ball of radius $\simeq h_T$, and $\norm{L^\infty(T)}{\nabla b_T}\lesssim h_T^{-1}$}.
  \end{aligned}
\end{equation}
If the triangular submesh $\Sh$ has at least one vertex inside $T$, $b_T$ can be defined as the hat function on $\Sh$ associated with this vertex.
Otherwise, taking a ball $B_T\subset T$ of radius $\simeq h_T$ (such a ball exists by mesh regularity assumption), $b_T$ can be taken as $h_T^{-1}$ times the distance to $\partial B_T$ inside $B_T$ and as $0$ outside $B_T$.
We then define $L^2_h\discrete{v}_h\in V$ such that, for all $T\in\Th$,
\begin{equation}\label{eq:def.L2h}
  \begin{aligned}
    &(L^2_h\discrete{v}_h)_{|T}\coloneq b_T q_{\discrete{v},T}\text{ where $q_{\discrete{v},T}\in\Poly{k}(T)$ is such that}\\
    &\int_T b_T q_{\discrete{v},T} r = \int_T (\RTL\discrete{v}_T - L^1_h\discrete{v}_h)r\qquad\forall r\in\Poly{k}(T).
  \end{aligned}
\end{equation}
The existence and uniqueness of $q_{\discrete{v},T}$ follow from the Riesz representation theorem applied in $\Poly{k+1}(T)$ equipped with the inner product $\langle \alpha,\beta\rangle_{b,T}\coloneq \int_T b_T  \alpha\beta$, which induces a norm equivalent, uniformly in $T$ and $h$, to the $L^2(T)$-norm (see the arguments in the proof of \cite[Lemma 1.25, point (i)]{Di-Pietro.Droniou:20}).
The property $L^2_h\discrete{v}_h\in V$ follows from the vanishing of each $b_T$ on $\partial T$, which ensures inter-element continuity.
Notice also that, if $b_T$ can be taken in $\Poly{1}_c(\Sh)$ as described above, then $L^2_h\discrete{v}_h\in\Poly{k+1}_{c,0}(\Sh)$.

\begin{lemma}[Properties of $L_h$]\label{lem:application:lifting}
  The lifting $L_h$ defined by \eqref{eq:application:Lh} satisfies Assumption \ref{ass:conforming.lifting}.
\end{lemma}

\begin{proof}
  \emph{Proof of \eqref{eq:Lh:consistency:ell}.}
  By \eqref{eq:def.L2h} we have, for all $T\in\Th$ and all $r\in\Poly{k}(T)$,
  \[
  \int_T L_h\discrete{v}_h\, r = \int_T (L^1_h\discrete{v}_h+L^2_h\discrete{v}_h)\, r = \int_T \RTL \discrete{v}_T\, r,
  \]
  which, recalling that $\Pi^L_T = \lproj{T}{k}$, shows $\Pi^L_T (L_h \discrete{v}_h)_{|T} = \RTL \discrete{v}_T$ as required.

  We also note a useful consequence of this equality. Recalling that $\RTL=\Pi^L_T\circ\dofmap_T^{-1}$, we have $\Pi^L_T(L_h\discrete{v}_h)=\Pi^L_T(\dofmap_T^{-1}\discrete{v}_h)$ and thus, since $\lproj{T}{k-1}\circ\Pi^L_T=\lproj{T}{k-1}\circ\lproj{T}{k}=\lproj{T}{k-1}$,
  \begin{equation}\label{eq:proj.k-1}
    \lproj{T}{k-1}(L_h\discrete{v}_h)=\lproj{T}{k-1}(\dofmap_T^{-1}\discrete{v}_h)\overset{\eqref{eq:application:sigmah}}= v_T.
  \end{equation}
  \\
  \emph{Proof of \eqref{eq:Lh:consistency:a}.}
  Let $T \in \Th$ and $w \in \Poly{k+1}(T)$.
  Using the definition \eqref{eq:application:aT} of $\discrete{a}_T$ together with the polynomial consistency \eqref{eq:application:GT.RaT:polynomial.consistency} of $G_T$ and $R^a_T$, we infer, for all $\discrete{v}_T\in\discrete{V}_T$,
  \[
  \discrete{a}_T(\dofmap_T w, \discrete{v}_T) = \int_T \nabla w\cdot G_T\discrete{v}_T
  \overset{\eqref{eq:def.GT}}= - \int_T v_T (\operatorname{div} \nabla w)
  + \int_{\partial T} \trrec{\partial T} \discrete{v}_T \, (\nabla w \cdot n_T).
  \]
  We then note that, by \eqref{eq:def.L1h} and since the values at Lagrange nodes entirely determine polynomials of degree $\le k+1$ on each edge, we have $\trrec{\partial T} \discrete{v}_T=(L^1_h\discrete{v}_h)_{|\partial T}=(L_h\discrete{v}_h)_{|\partial T}$,
  where the last equality comes from \eqref{eq:def.L2h} and \eqref{eq:def.bT} which show that $(L^2_h\discrete{v}_h)_{|\partial T} =0$.
  Using \eqref{eq:proj.k-1} and the fact that $\operatorname{div} \nabla w\in\Poly{k-1}(T)$, we therefore get
  \[
  \discrete{a}_T(\dofmap_T w, \discrete{v}_T) = - \int_T L_h\discrete{v}_h\, (\operatorname{div} \nabla w)
  + \int_{\partial T} L_h\discrete{v}_h \, (\tau \cdot n_T)=\int_T \nabla w\cdot\nabla (L_h\discrete{v}_h)
  =a_{|T}(w,(L_h\discrete{v}_h)_{|T}),
  \]
  where we have performed an integration by parts in the second step. This concludes the proof of \eqref{eq:Lh:consistency:a}.
  \medskip\\
  \emph{Proof of \eqref{eq:Lh:boundedness}.}
  Let $T\in\Th$ and use a triangle inequality to write
  \begin{align}
    \norm{L^2(T;\Real^2)}{\nabla L_h\discrete{v}_h}^2
    &\le 2\norm{L^2(T;\Real^2)}{\nabla (L_h\discrete{v}_h-\RTa\discrete{v}_T)}^2
    + 2\norm{L^2(T;\Real^2)}{\nabla \RTa\discrete{v}_T}^2
    \nonumber\\
    \overset{\eqref{eq:application:RTa:boundedness}}&\lesssim
    \norm{L^2(T;\Real^2)}{\nabla (L_h\discrete{v}_h-\RTa\discrete{v}_T)}^2
    + \norm{\discrete{V}_T}{\discrete{v}_T}^2.
    \label{eq:application:bound.Lh}
  \end{align}
  To estimate the first term in the right-hand side above, we recall \eqref{eq:application:Lh} and use a triangle inequality to write
  \begin{align}
    \norm{L^2(T;\Real^2)}{\nabla (L_h\discrete{v}_h-\RTa\discrete{v}_T)}^2
    &\le 2\norm{L^2(T;\Real^2)}{\nabla (L^1_h\discrete{v}_h-\RTa\discrete{v}_T)}^2+2\norm{L^2(T;\Real^2)}{\nabla L^2_h\discrete{v}_h}^2\nonumber\\
    \overset{\eqref{eq:def.bT}}&\lesssim h_T^{-2}\norm{L^2(T)}{L^1_h\discrete{v}_h-\RTa\discrete{v}_T}^2+h_T^{-2}\norm{L^2(T)}{q_{\discrete{v},T}}^2\nonumber\\
    &\lesssim h_T^{-2}\norm{L^2(T)}{L^1_h\discrete{v}_h-\RTa\discrete{v}_T}^2  +h_T^{-2}\norm{L^2(T)}{\RTL\discrete{v}_T-L^1_h\discrete{v}_h}^2\nonumber\\
    &\lesssim h_T^{-2}\norm{L^2(T)}{L^1_h\discrete{v}_h-\RTa\discrete{v}_T}^2  +h_T^{-2}\norm{L^2(T)}{\RTL\discrete{v}_T-\RTa\discrete{v}_T}^2
    \label{eq:application:bound.Lh.RTa}
  \end{align}
  where, in the second line, we have additionally used H\"older inequalities on $(\nabla L^2_h\discrete{v}_h)_{|T}=b_T \nabla q_{\discrete{v},T}+q_{\discrete{v},T}\nabla b_T$ and inverse inequalities on the piecewise polynomial functions (see \cite[Remark 1.33]{Di-Pietro.Droniou:20}) $L^1_h\discrete{v}_h-\RTa\discrete{v}_T$ and $q_{\discrete{v},T}$, while, in the third line, we have used \eqref{eq:def.L2h} with $r=q_{\discrete{v},T}$ and the uniform equivalence of the $L^2(T)$-norm and the norm induced by the inner product $\langle\cdot,\cdot\rangle_{b,T}$. The conclusion was obtained writing $\RTL\discrete{v}_T-L^1_h\discrete{v}_h=(\RTL\discrete{v}_T-\RTa\discrete{v}_h)+(\RTa\discrete{v}_T-L^1_h\discrete{v}_h)$ and using a triangle inequality.

  A standard $L^2$- and DOF-norm equivalence for finite element functions in $\Poly{k+1}_c(\ST)$ (where $\ST$ is the local simplicial mesh obtained by considering the simplices of $\Sh$ contained in $T$) gives, with $\mathfrak N_T$ the Lagrange nodes in $\ST$,
  \begin{align}
    h_T^{-2}\norm{L^2(T)}{L^1_h\discrete{v}_h-\RTa\discrete{v}_T}^2&\lesssim
    h_T^{-2}\sum_{P\in\mathfrak N_T} |T|\,|L^1_h\discrete{v}_h(x_P)-\RTa\discrete{v}_T(x_P)|^2\nonumber\\
    \overset{\eqref{eq:def.L1h}}&=
    h_T^{-2}\sum_{P\in\mathfrak N_T\cap \partial T} |T|\,|\trrec{\partial T}\discrete{v}_h(x_P)-\RTa\discrete{v}_T(x_P)|^2\nonumber\\
    &\lesssim h_T^{-2}\sum_{P\in\mathfrak N_T\cap \partial T} |T|\,|\partial T|^{-1}\norm{L^2(\partial T)}{\trrec{\partial T}\discrete{v}_h-\RTa\discrete{v}_T}^2\nonumber\\
    &\lesssim h_T^{-1}\norm{L^2(\partial T)}{\trrec{\partial T}\discrete{v}_h-\RTa\discrete{v}_T}^2
    \overset{\eqref{eq:application:discrete.norm}}\lesssim \norm{\discrete{V}_T}{\discrete{v}_T}^2,
    \label{eq:application:est.L1h}
  \end{align}
  where the second inequality follows from an inverse Lebesgue inequality \cite[Lemma 1.25]{Di-Pietro.Droniou:20}, while the third inequality is a consequence of the mesh regularity property which ensures that $|T|\le h_T|\partial T|$.

  To estimate the last term in \eqref{eq:application:bound.Lh.RTa}, we write
  \begin{equation}
    h_T^{-2}\norm{L^2(T)}{\RTL\discrete{v}_T-\RTa\discrete{v}_h}^2\overset{\eqref{eq:RTL=pikT.RaT}}= h_T^{-2}\norm{L^2(T)}{\lproj{T}{k}(\RTa\discrete{v}_T)-\RTa\discrete{v}_T}^2
    \lesssim \norm{L^2(T;\Real^2)}{\nabla \RTa\discrete{v}_T}^2\overset{\eqref{eq:application:RTa:boundedness}}\lesssim \norm{\discrete{V}_T}{\discrete{v}_T}^2,
    \label{eq:application:bound.RTa}
  \end{equation}
  where the first inequality stems from the approximation properties of $\lproj{T}{k}$ (see, e.g., \cite[Theorem 1.45]{Di-Pietro.Droniou:20}).

  The proof of the bound \eqref{eq:Lh:boundedness} is completed by plugging \eqref{eq:application:est.L1h} and \eqref{eq:application:bound.RTa} into \eqref{eq:application:bound.Lh.RTa}, by using the resulting estimate in \eqref{eq:application:bound.Lh}, and by summing over $T\in\Th$.
\end{proof}

\begin{theorem}[Error estimate based on a conforming lifting]\label{thm:error.estimates.via.lifting}
  Assume that the solution to \eqref{eq:weak} satisfies $u\in H^{k+2}(\Th)$, that $f\in H^{k+1}(\Th)$, and that $\discrete{u}_h\in\discrete{V}_h$ is the solution to \eqref{eq:fully:discrete} with the local bilinear forms and global linear form defined by \eqref{eq:application:aT} and \eqref{eq:application:ell.h:RTL:discrete}, respectively.
  Then, it holds
  \begin{equation}\label{eq:error.est.via.lifting}
    \norm{\discrete{V}_h}{\discrete{u}_h-\dofmap_h u}\lesssim h^{k+1}\left(\seminorm{H^{k+2}(\Th)}{u}+\seminorm{H^{k+1}(\Th)}{f}\right).
  \end{equation}
\end{theorem}

\begin{proof}
  By Lemma \ref{lem:application:lifting} we can apply Theorem \ref{thm:discrete:error.estimate}. The infimum is bounded above by the choice $w=\lproj{T}{k+1}u$ and the contributions involving $\seminorm{H^1(T)}{u-\lproj{T}{k+1}u}$ and $\norm{L^2(T)}{f-\lproj{T}{k}f}$ in \eqref{eq:discrete:error.estimate} can be directly estimated using the approximation properties of $L^2(T)$-orthogonal projectors on polynomial spaces (see \cite[Theorem 1.45]{Di-Pietro.Droniou:20}). It remains to bound $\norm{\discrete{V}_T}{\dofmap_T(u-\lproj{T}{k+1}u)}$ for all $T\in\Th$.
  To this purpose, we write
  \[
  \begin{aligned}
    \norm{\discrete{V}_T}{\dofmap_T(u-\lproj{T}{k+1}u)}
    \overset{\eqref{eq:application:discrete.norm},\eqref{eq:application:GT.RaT:polynomial.consistency}}&= \Big(\norm{L^2(T;\Real^2)}{G_T(\dofmap_T u)-\nabla\lproj{T}{k+1}u}^2
    \\
    &\quad + h_T^{-1}\norm{L^2(\partial T)}{(\RTa(\dofmap_T u)-\lproj{T}{k+1}u)-\trrec{\partial T}\dofmap_T(u-\lproj{T}{k+1}u)}^2\Big)^{\frac12}
    \\
    \overset{\eqref{eq:application:PiTa}}&\lesssim
    \norm{L^2(T;\Real^2)}{G_T(\dofmap_T u)-\nabla\lproj{T}{k+1}u}  +h_T^{-\frac12}\norm{L^2(\partial T)}{\Pi^a_T u-\lproj{T}{k+1}u}\\
    &\quad + h_T^{-\frac12}\sum_{E\in\ET}\norm{L^2(E)}{\lproj{E}{k-1}(u-\lproj{T}{k+1}u)}
    +\sum_{V\in\VT}|u(x_V)-(\lproj{T}{k+1}u)(x_V)|,
  \end{aligned}
  \]
  where, in the conclusion, we have additionally used \eqref{eq:estim.poly.edge} on each $\trrec{\partial T}\dofmap_T(u-\lproj{T}{k+1}u)_{|E}$ for $E\in\ET$ together with the definitions of $\trrec{\partial T}$ and $\sigma_T$, which show that, for all $v\in C(\overline{T})$,
  $\lproj{E}{k-1}\trrec{E}(\sigma_T v)=\lproj{E}{k-1} v$ for all $E\in\ET$
  and $\trrec{E}(\sigma_T v)(x_V)=v(x_V)$ for all $V\in\VT$.
  A discrete trace inequality on $\Pi^a_T u-\lproj{T}{k+1}u$ and the $L^2(E)$-boundedness of $\lproj{E}{k-1}$ then yield
  \[
  \begin{aligned}
    \norm{\discrete{V}_T}{\dofmap_T(u-\lproj{T}{k+1}u)}
    &\lesssim
    \norm{L^2(T;\Real^2)}{G_T(\dofmap_T u)-\nabla\lproj{T}{k+1}u} +h_T^{-1}\norm{L^2(T)}{\Pi^a_T u-\lproj{T}{k+1}u}\\
    &\quad
    +h_T^{-\frac12}\sum_{E\in\ET}\norm{L^2(E)}{u-\lproj{T}{k+1}u}
    +\sum_{V\in\VT}|u(x_V)-(\lproj{T}{k+1}u)(x_V)|.
  \end{aligned}
  \]
  The approximation properties of $G_T$, $\Pi^a_T$ and $\lproj{T}{k+1}$ (see Lemmas \ref{lem:approx.GT} and \ref{lem:approx.PiTa}, and \cite[Theorem 1.45]{Di-Pietro.Droniou:20}) show that the right-hand side about is $\lesssim h_T^{k+1}\seminorm{H^{k+2}(T)}{u}$, which concludes the proof.
\end{proof}

\subsection{Estimate of the consistency error by direct manipulations}\label{sec:application:direct.manipulations}

In this section we showcase the last technique to derive estimates of the consistency error, based on direct manipulations of it.
Specifically, we provide a direct proof of the following lemma, which, combined with Theorem \ref{thm:discrete:error.estimate}, gives a direct proof of the error estimate stated in Theorem \ref{thm:error.estimates.via.lifting}.

\begin{lemma}[Estimate of the consistency error]\label{lem:Eh:discrete.norm.estimate}
  Denote by $u$ the solution to the variational problem \eqref{eq:weak} with spaces and forms defined by \eqref{eq:poisson} and let $\mathfrak{E}_h : \discrete{V}_h \to \Real$ be the linear form such that
  \begin{equation}\label{eq:application:Eh}
    \mathfrak{E}_h(\discrete{v}_h)
    \coloneqq -\sum_{T \in \Th} \int_T (\operatorname{div} \nabla u) \, \RTa \discrete{v}_T
    - \discrete{a}_h(\dofmap_h u, \discrete{v}_h)
    \qquad \forall \discrete{v}_h \in \discrete{V}_h.
  \end{equation}
  Then, assuming $u \in H^{k+2}(\Th)$, it holds
  \begin{equation}\label{eq:application:tEh:discrete.norm.estimate}
    \norm{\discrete{V}_h'}{\mathfrak{E}_h}
    \lesssim h^{k+1} \seminorm{H^{k+2}(\Th)}{u}.
  \end{equation}
  As a consequence, additionally assuming $f \in H^k(\Th)$,
  \begin{equation}\label{eq:application:Eh:discrete.norm.estimate}
    \norm{\discrete{V}_h'}{\mathcal{E}_h}
    \lesssim h^{k+1} \left(
    \seminorm{H^{k+2}(\Th)}{u} + \seminorm{H^k(\Th)}{f}
    \right).
  \end{equation}
\end{lemma}

\begin{remark}[Regularity of loading term]\label{rem:load:ort}
  The estimate \eqref{eq:application:Eh:discrete.norm.estimate} only requires $f \in H^k(\Th)$, as opposed to $f \in H^{k+1}(\Th)$ in \eqref{eq:vem:conv} and \eqref{eq:error.est.via.lifting}.
  To lower the regularity requirement in these estimates, obtained through Theorems \ref{thm:vem:error.estimate} and \ref{thm:discrete:error.estimate}, the following  assumption should be added to the abstract frameworks:
  \[
  \norm{L_{|T}}{w_T-\Pi^L_Tw_T}\lesssim h_T \norm{V_{|T}}{w_T}\qquad\forall T\in\Th,
  \]
  for all $w_T\in\virtual{V}_T$ in the virtual framework of Section \ref{sec:abstract.framework:vem}, or all $w_T \in (L_h \discrete{V}_h)_{|T}$ for the fully discrete framework of Section \ref{sec:abstract.framework:fully.discrete}.
  Then, using the same trick as in \eqref{eq:handling.rhs} below, we could replace the terms $\norm{L_{|T}}{f-\Pi^L_Tf}^2$ in \eqref{eq:vem:error.estimate} and \eqref{eq:discrete:error.estimate} by $h_T^2\norm{L_{|T}}{f-\Pi^L_Tf}^2$, and $f \in H^k(\Th)$ would then be sufficient to obtain an estimate in $h_T^{2(k+1)}$ for this term.
\end{remark}

\begin{remark}[Alternative choice of the forcing term]\label{rem:rhs}
  The estimate \eqref{eq:application:tEh:discrete.norm.estimate} indicates that we can obtain an optimally convergent scheme also with the following discretization of the forcing term:
  \begin{equation}\label{eq:application:ell.h:RTa:discrete}
    \ell_h(\discrete{v}_h)
    = \sum_{T \in \Th} \int_T f \, \RTa \discrete{v}_T
    \qquad \forall \discrete{v}_T \in \discrete{V}_T,
  \end{equation}
  where the difference with respect to \eqref{eq:application:ell.h:RTL:discrete} is that the reconstruction $\RTa$ is used instead of $\RTL$.
  With this choice, the consistency error of \eqref{eq:fully:discrete} simply is $\mathcal E_h=\mathfrak E_h$, and combining \eqref{eq:application:tEh:discrete.norm.estimate} with Theorem \ref{thm:error.estimate} provides an error estimate without requiring any regularity on $f$. The natural choice \eqref{eq:application:ell.h:RTa:discrete} is, however, not immediately suggested by the abstract analysis frameworks of Section \ref{sec:framework}, incidentally showing the interest of direct estimates also at the design stage.

  An increased flexibility on the forcing term can be additionally used, e.g., to reproduce physically relevant properties such as flux conservation (see \cite[Remark 2.22]{Di-Pietro.Droniou:20}) or pressure-robustness (see \cite{Linke:14} and also \cite{Beirao-da-Veiga.Dassi.ea:22} concerning polytopal methods linked to the one presented here).
\end{remark}

\begin{proof}[Proof of Lemma \ref{lem:Eh:discrete.norm.estimate}]
  By \eqref{eq:application:RTa} with $\tau = G_T(\dofmap_T u)$ (see Remark \ref{eq:application:RTa:validity}), it holds, for all $\discrete{v}_h \in \discrete{V}_h$,
  \[
  \sum_{T \in \Th} \left[
    \int_T \operatorname{div} (G_T(\dofmap_T u)) \, \RTa \discrete{v}_T
    + \int_T G_T(\dofmap_T u) \cdot G_T \discrete{v}_T
    - \int_{\partial T} (G_T(\dofmap_T u) - \nabla u) \cdot n_T \, \trrec{\partial T} \discrete{v}_T
    \right]
  = 0,
  \]
  where we have additionally used the fact that the normal component of $\nabla u$ is continuous across interfaces together with the fact that $\trrec{E} \discrete{v}_E = 0$ for all $E \in \Eh^{\rm b}$ (since the boundary components of $\discrete{v}_h \in \discrete{V}_h$ vanish) to insert $\nabla u \cdot n_T$ into the last term.
  Integrating by parts the first term, we get
  \begin{multline*}
    \sum_{T \in \Th} \Bigg[
      - \int_T G_T(\dofmap_T u) \cdot \nabla \RTa \discrete{v}_T
      + \sum_{T \in \Th} \int_{\partial T} (G_T(\dofmap_T u) \cdot n_T) \RTa \discrete{v}_T
      \\
      + \int_T G_T(\dofmap_T u) \cdot G_T \discrete{v}_T
      - \int_{\partial T} (G_T(\dofmap_T u) - \nabla u) \cdot n_T \, \trrec{\partial T} \discrete{v}_T
      \Bigg]
    = 0.
  \end{multline*}
  Adding this expression to \eqref{eq:application:Eh} after integrating by parts element-wise the first term in the right-hand side,
  and expanding $\discrete{a}_h$ and then each $\discrete{a}_T$, $T \in \Th$, according to the respective definitions \eqref{eq:discrete:ah} and \eqref{eq:application:aT}, we get
  \begin{multline*}
    \mathfrak{E}_h(\discrete{v}_h)
    = \sum_{T \in \Th} \left[
      \int_T (\nabla u - G_T(\dofmap_T u)) \cdot \nabla \RTa \discrete{v}_T
      + \int_{\partial T} (G_T(\dofmap_T u)-\nabla u) \cdot n_T \, (\RTa \discrete{v}_T - \trrec{\partial T} \discrete{v}_T)
      \right]
    \\
    - \sum_{T \in \Th} \discrete{s}_T(\dofmap_T (u - \Pi_T^a u), \discrete{v}_T - \dofmap_T \RTa \discrete{v}_T).
  \end{multline*}
  Applying Cauchy--Schwarz and continuous trace inequalities, we obtain
  \[
  \begin{aligned}
    \mathfrak{E}_h(\discrete{v}_h)
    &\lesssim \left[
      \sum_{T \in \Th} \left(
      \norm{L^2(T;\Real^2)}{\nabla u - G_T(\dofmap_T u)}^2
      + h_T^2 \seminorm{H^1(T;\Real^2)}{\nabla u - G_T(\dofmap_T u)}^2
      \right)
      \right]^{\frac12}
    \\
    &\quad \times \left[
      \sum_{T \in \Th} \left(
      \norm{L^2(T;\Real^2)}{\nabla \RTa \discrete{v}_T}^2
      + h_T^{-1} \norm{L^2(\partial T)}{\RTa \discrete{v}_T - \trrec{\partial T} \discrete{v}_T}^2
      \right)
      \right]^{\frac12}
    \\
    &\quad + \left(
    \sum_{T \in \Th} \discrete{s}_T(\dofmap_T (u - \Pi_T^a u), \dofmap_T (u - \Pi_T^a u))
    \right)^{\frac12} \norm{\discrete{V}_h}{\discrete{v}_h}
    \\
    \overset{\eqref{eq:application:GT:approximation},\eqref{eq:application:RTa:boundedness},\eqref{eq:application:discrete.norm},\eqref{eq:application:sT:consistency},\eqref{eq:application:ah:coercivity.boundedness}}&\lesssim
    h^{k+1} \seminorm{H^2(\Th)}{u} \norm{\discrete{V}_h}{\discrete{v}_h}.
  \end{aligned}
  \]
  Passing to the supremum over $\discrete{v}_h \in \discrete{V}_h \setminus \{ \discrete{0}\}$ after having divided by $\norm{\discrete{V}_h}{\discrete{v}_h}$, \eqref{eq:application:tEh:discrete.norm.estimate} follows.

  We next notice that, for all $\discrete{v}_h \in \discrete{V}_h$,
  \begin{equation}\label{eq:handling.rhs}
    \begin{aligned}
      \mathcal{E}_h(\discrete{v}_h)
      &= \mathfrak{E}_h(\discrete{v}_h)
      - \sum_{T \in \Th} \int_T f (\RTa \discrete{v}_T - \RTL \discrete{v}_T)
      \\
      \overset{\eqref{eq:RTL=pikT.RaT}}&=
      \mathfrak{E}_h(\discrete{v}_h)
      - \sum_{T \in \Th} \int_T (f - \lproj{T}{k} f) (\RTa \discrete{v}_T - \lproj{T}{k} \RTa \discrete{v}_T)
      \\
      &\le
      \mathfrak{E}_h(\discrete{v}_h)
      + \left(
      \sum_{T \in \Th} \norm{L^2(T)}{f - \lproj{T}{k} f}^2\right)^{\frac12}
      \left(
      \sum_{T \in \Th} \norm{L^2(T)}{\RTa \discrete{v}_T - \lproj{T}{k} \RTa \discrete{v}_T}^2
      \right)^{\frac12}
      \\
      &\le
      \mathfrak{E}_h(\discrete{v}_h)
      + h^{k+1} \norm{H^k(\Th)}{f} \norm{\discrete{V}_h}{\discrete{v}_h},
    \end{aligned}
  \end{equation}
  where we have additionally used the orthogonality properties of $\lproj{T}{k}$ in the second step, Cauchy--Schwarz inequalities in the third step,
  while the conclusion follows using local approximation estimates for $\lproj{T}{k}$ together with a local Poincar\'e inequality and the $H^1$-boundedness of $\lproj{T}{k}$ to write ${\norm{L^2(T)}{\RTa \discrete{v}_T - \lproj{T}{k} \RTa \discrete{v}_T}} \lesssim h_T \norm{L^2(T)^2}{\nabla \RTa \discrete{v}_T} \lesssim h_T \norm{\discrete{V}_T}{\discrete{v}_T}$.
  Passing to the dual norm concludes the proof of \eqref{eq:application:Eh:discrete.norm.estimate}.
\end{proof}


\appendix

\section{A scaled trace lemma}\label{app:scaled.trace}

In this section we show the proof of the following trace lemma.

\begin{lemma}\label{lem:scal:trace}
  Let $T \in \Th$, satisfying the mesh assumptions in Section \ref{sec:mesh}. Recalling that $\tnorm{H^{-\frac12}(\partial T)}{\cdot}$ represents the dual norm to $\tnorm{H^{\frac12}(\partial T)}{\cdot}=h_T^{-\frac12}\norm{L^2(\partial T)}{\cdot}+\seminorm{H^{\frac12}(\partial T)}{\cdot}$, it holds
  $$
  \tnorm{H^{-\frac12}(\partial T)}{{\bf w}\cdot n_T}
  \lesssim
  \norm{L^2(T)}{{\bf w}}
  + h_T \norm{L^2(T)}{{\rm div}\,{\bf w}}
  \qquad \forall {\bf w} \in H_{\rm div}(T).
  $$
\end{lemma}
\begin{proof}
  By a trivial extension of the proof of Lemma 6.1 in \cite{Beirao-da-Veiga.Lovadina.ea:17}, it follows that for each $\varphi \in H^{1/2}(\partial T)$ there exists a lifting $\widetilde{\varphi} \in H^{1}(T)$ (i.e. $\widetilde{\varphi}|_{\partial T} = \varphi$) satisfying
  \begin{equation*} 
    h_T^{-1} \norm{L^2(T)}{\widetilde{\varphi}}
    + \seminorm{H^1(T)}{\widetilde{\varphi}}
    \lesssim
    \tnorm{H^{\frac12}(\partial T)}{\varphi},
  \end{equation*}
  with hidden constant only dependent on the shape regularity constant of $T$.
  Using first the definition of the dual norm, then introducing the above lifting and applying an integration by parts, we obtain
  $$
  \begin{aligned}
    \tnorm{H^{-\frac12}(\partial T)}{{\bf w}\cdot n_T} & =
    \sup_{\varphi \in H^{\frac12}(\partial T)} \frac{\langle{\bf w} \cdot n_T,\varphi \rangle_{\partial T}}{\tnorm{H^{\frac12}(\partial T)}{\varphi}}
    \lesssim \sup_{\varphi \in H^{\frac12}(\partial T)} \frac{\langle{\bf w} \cdot n_T,\widetilde{\varphi} \rangle_{\partial T}}{h_T^{-1} \norm{L^2(T)}{\widetilde{\varphi}}
      + \seminorm{H^1(T)}{\widetilde{\varphi}}}  \\
    & = \sup_{\varphi \in H^{\frac12}(\partial T)} \frac{\int_{T} ({\bf w} \cdot \nabla \widetilde{\varphi} + \widetilde{\varphi} \, {\rm div}{\bf w} )}{h_T^{-1} \norm{L^2(T)}{\widetilde{\varphi}}
      + \seminorm{H^1(T)}{\widetilde{\varphi}}}
    \, ,
  \end{aligned}
  $$
  where the suprema are taken on non vanishing functions.
  Cauchy-Schwarz inequalities conclude the proof.
\end{proof}

\section{Adjoint consistency error}\label{sec:adjoint.consistency}

Let us consider the following portion of an $L^2$-Hilbert complex:
\begin{equation}\label{eq:complex:continuous}
  \begin{tikzcd}[sep=large]
    V^{i-1} \arrow[r,"d^{i-1}"]
    & V^{i-1} \arrow[r,"d^i"]
    & V^{i+1}
  \end{tikzcd}
\end{equation}
as well as its approximation in the spirit of Section \ref{sec:abstract.framework:fully.discrete} depicted in the following diagram:
\begin{equation}\label{eq:complex}
  \begin{tikzcd}[sep=large]
    V_I^{i-1} \arrow[d,"\dofmap_h^{i-1}"]
    & V_I^i \arrow[d,"\dofmap_h^i"]
    & V_I^{i+1} \arrow[d,"\dofmap_h^{i+1}"]
    \\
    \discrete{V}_h^{i-1} \arrow[r,"\discrete{d}_h^{i-1}"]
    & \discrete{V}_h^i \arrow[r,"\discrete{d}_h^i"]
    & \discrete{V}_h^{i+1}.
  \end{tikzcd}
\end{equation}
Above, $V_I \subset V$ denotes the domain of the interpolator and the graded map $\discrete{d}_h$ is the discrete differential, for which we assume the following commutation property with the continuous differential:
For all $v \in V_I^{i-1}$ such that $d^{i-1} v \in V_I^i$,
\begin{equation}\label{eq:complex:cochain.map}
  \discrete{d}_h^{i-1} \dofmap_h^{i-1} v
  = \dofmap_h^i d^{i-1} v.
\end{equation}

Assume, for the sake of simplicity, that the space of harmonic $i$-forms is trivial.
The mixed formulation of the Hodge Laplacian for the portion of the continuous complex \eqref{eq:complex:continuous} consists in finding $(u, p) \in V^{i-1} \times V^i$ such that
\begin{equation}\label{eq:complex:hodge.laplacian:continuous}
  \begin{alignedat}{2}
    (u, v)_{i-1} - (p, d^{i-1} v)_i &= 0
    &\qquad& \forall v \in V^{i-1},
    \\
    (d^{i-1} u, q)_i
    + (d^i p, d^i q)_{i+1}
    &= (g, q)_i
    &\qquad& \forall q \in V^i,
  \end{alignedat}
\end{equation}
where $(\cdot,\cdot)_j$ denotes the $L^2$-product for $V^j$ and
$g$ is a forcing term with appropriate regularity.
Let each space $\discrete{V}_h^j$ be equipped with a discrete $L^2$-product $(\cdot,\cdot)_{j,h}$.
The discretization of \eqref{eq:complex:hodge.laplacian:continuous} based on the discrete complex depicted in \eqref{eq:complex} then reads:
Find $(\discrete{u}_h, \discrete{p}_h) \in \discrete{V}_h^{i-1} \times \discrete{V}_h^i$ such that
\begin{subequations}\label{eq:complex:hodge.laplacian}
  \begin{alignat}{2}
    \label{eq:complex:hodge.u.deltap}
    (\discrete{u}_h, \discrete{v}_h)_{i-1,h}
    - (\discrete{p}_h, \discrete{d}_h^{i-1} \discrete{v}_h)_{i,h}
    &= 0
    &\qquad& \forall \discrete{v}_h \in \discrete{V}_h^{i-1},
    \\
    (\discrete{d}_h^{i-1} \discrete{u}_h, \discrete{q}_h)_{i,h}
    + (\discrete{d}_h^i \discrete{p}_h, \discrete{d}_h^i \discrete{q}_h)_{i+1,h}
    &= \ell_h(\discrete{q}_h)
    &\qquad& \forall \discrete{q}_h \in \discrete{V}_h^i.
  \end{alignat}
\end{subequations}
An equivalent variational formulation is:
Find $(\discrete{u}_h, \discrete{p}_h) \in \discrete{V}_h^{i-1} \times \discrete{V}_h^i$ such that
\[
\discrete{a}_h( (\discrete{u}_h, \discrete{p}_h), (\discrete{v}_h, \discrete{q}_h) )
= \ell_h(\discrete{q}_h)
\qquad \forall (\discrete{v}_h, \discrete{q}_h) \in \discrete{V}_h^{i-1} \times \discrete{V}_h^i,
\]
with bilinear form $\discrete{a}_h : \left[ \discrete{V}_h^{i-1} \times \discrete{V}_h^i \right]^2 \to \Real$ such that, for all $(\discrete{w}_h, \discrete{r}_h), (\discrete{v}_h, \discrete{q}_h) \in \discrete{V}_h^{i-1} \times \discrete{V}_h^i$,
\[
\discrete{a}_h( (\discrete{w}_h, \discrete{r}_h), (\discrete{v}_h, \discrete{q}_h) )
\coloneqq
(\discrete{w}_h, \discrete{v}_h)_{i-1,h}
- (\discrete{r}_h, \discrete{d}_h^{i-1} \discrete{v}_h)_{i,h}
+ (\discrete{d}_h^{i-1} \discrete{w}_h, \discrete{q}_h)_{i,h}
+ (\discrete{d}_h^i \discrete{r}_h, \discrete{d}_h^i \discrete{q}_h)_{i+1,h}.
\]
Assuming the additional regularity $(u,p) \in V_I^{i-1} \times V_I^i$ for the solution to \eqref{eq:complex:hodge.laplacian:continuous}, the consistency error defined according to \eqref{eq:consistency.error} reads
\[
\mathcal{E}_h(\discrete{v}_h, \discrete{q}_h)
= \ell_h(\discrete{q}_h)
- \discrete{a}_h( (\dofmap_h^{i-1} u, \dofmap_h^i p), (\discrete{v}_h, \discrete{q}_h) ).
\]

Assume now that $g$ is smooth enough for the following choice to make sense:
\[
\ell_h(\discrete{q}_h)
= (\dofmap_h^i g, \discrete{q}_h)_{i,h}
\qquad \forall \discrete{q}_h \in \discrete{V}_h^i.
\]
Then, recalling that $g = d^{i-1} u + \delta^{i+1} d^i p$ almost everywhere, with $\delta$ denoting the codifferential, and using the commutation property \eqref{eq:complex:cochain.map} to perform the cancellations, we have
\[
\begin{aligned}
  \mathcal{E}_h(\discrete{v}_h, \discrete{q}_h)
  &= \cancel{(\dofmap_h^i d^{i-1} u, \discrete{q}_h)_{i,h}}
  + (\dofmap_h^i \delta^{i+1} d^i p, \discrete{q}_h)_{i,h}
  \\
  &\quad
  - (\dofmap_h^{i-1} u, \discrete{v}_h)_{i-1,h}
  + (\dofmap_h^i p, \discrete{d}_h^{i-1} \discrete{v}_h)_{i,h}
  \\
  &\quad
  - \cancel{(\discrete{d}_h^{i-1} \dofmap_h^{i-1} u, \discrete{q}_h)_{i,h}}
  - (\discrete{d}_h^{i} \dofmap_h^i p, \discrete{d}_h^{i} \discrete{q}_h)_{i+1,h}
  \\
  &=
  (\dofmap_h^i \delta^{i+1} d^i p, \discrete{q}_h)_{i,h}
  - (\dofmap_h^{i-1} \delta^i p, \discrete{v}_h)_{i-1,h}
  + (\dofmap_h^i p, \discrete{d}_h^{i-1} \discrete{v}_h)_{i,h}
  - (\dofmap_h^{i+1}d^i p, \discrete{d}_h^{i} \discrete{q}_h)_{i+1,h}
  \\
  &=
  \widetilde{\mathcal{E}}_h^i(d^i p; \discrete{q}_h)
  - \widetilde{\mathcal{E}}_h^{i-1}(p; \discrete{v}_h),
\end{aligned}
\]
where, in the second equality, we have used \eqref{eq:complex:cochain.map} again and the relation $u=\delta^i p$ (see \eqref{eq:complex:hodge.u.deltap}) and, in the last equality,
for any index $j$ and any $\omega \in V_I^{j+1}$ such that $\delta^{j+1} \omega \in V_I^j$, we have set
\[
\widetilde{\mathcal{E}}_h^j(\omega; \discrete{w}_h)
\coloneqq (\dofmap_h^j \delta^{j+1} \omega, \discrete{w}_h)_{j,h}
- (\dofmap_h^{j+1} \omega, \discrete{d}_h^j \discrete{w}_h)_{j+1,h}
\qquad \forall \discrete{w}_h \in \discrete{V}_h^j.
\]
The above quantity measures the failure to fulfill, at the discrete level, the relation expressing the adjointess of $\delta^{j+1}$ to $d^j$, and is therefore usually called \emph{adjoint consistency error} in the framework of DDR methods.
Estimates for the adjoint consistency errors associated with the DDR discretization of the three-dimensional de Rham complex were first proved in \cite{Di-Pietro.Droniou:23} based on direct manipulations for the gradient and the divergence and on the construction of a conforming lifting for the curl.
A unified estimate valid for any space dimension and form degree has been recently proved in \cite{Di-Pietro.Droniou.ea:25} based on a conforming lifting in the spirit of Section \ref{sec:abstract.framework:fully.discrete}.
Contrary to the original lifting of \cite{Di-Pietro.Droniou:23}, which was sought as the solution of a continuous PDE problem, the lifting of \cite{Di-Pietro.Droniou.ea:25} is obtained by direct prescription of the DOFs in a finite element space of appropriate degree built on a conforming simplicial submesh.

\begin{remark}[Link with Finite Element Systems]
  Recently, an interpretation of DDR methods in terms of Finite Element Systems \cite{Christiansen:08,Christiansen.Munthe-Kaas.ea:11} has been discussed in \cite{Christiansen.Rapetti:25}.
  The key idea of this work consists in defining harmonic extensions (a notion essentially akin to that of virtual functions) in such a way that, provided a norm equivalence of the form \eqref{eq:E} holds (cf. \cite[Eq. (4.11)]{Christiansen.Rapetti:25}), consistency of the discrete $L^2$-product can be proved in a spirit similar to Section \ref{sec:abstract.framework:vem}.
  This norm equivalence is however not proved, and the discussion in Section \ref{sec:nodal.virtual.analysis} shows that, even for the simplest gradient operator, it can be quite challenging (it is even more for the other operators in the de Rham complex).
\end{remark}


\section*{Acknowledgements}

Funded by the European Union (ERC Synergy, NEMESIS, project number 101115663).
Views and opinions expressed are however those of the authors only and do not necessarily reflect those of the European Union or the European Research Council Executive Agency. Neither the European Union nor the granting authority can be held responsible for them.


\printbibliography

@Article{         Achdou.Bernardi.ea:03,
  author        = {Achdou, Y. and Bernardi, C. and Coquel, F.},
  title         = {A priori and a posteriori analysis of finite volume
                  discretizations of {D}arcy's equations},
  journal       = {Numer. Math.},
  volume        = {96},
  year          = {2003},
  number        = {1},
  pages         = {17--42},
  doi           = {10.1007/s00211-002-0436-7}
}

@Article{         Ahmad.Alsaedi.ea:13,
  title         = {Equivalent projectors for virtual element methods},
  author        = {Ahmad, Bashir and Alsaedi, Ahmed and Brezzi, Franco and
                  Marini, L Donatella and Russo, Alessandro},
  journal       = {Comput. Math. Appl.},
  volume        = {66},
  number        = {3},
  pages         = {376--391},
  year          = {2013},
  publisher     = {Elsevier}
}

@Article{         Arnold.Falk.ea:06,
  title         = {Finite element exterior calculus, homological techniques,
                  and applications},
  volume        = {15},
  doi           = {10.1017/S0962492906210018},
  journal       = {Acta Numerica},
  author        = {Arnold, Douglas N. and Falk, Richard S. and Winther,
                  Ragnar},
  year          = {2006},
  pages         = {1–155}
}

@Article{         Beirao-da-Veiga.Brezzi.ea:13,
  title         = {Basic principles of virtual element methods},
  author        = {{Beir\~ao da Veiga}, Lourenco and Brezzi, Franco and
                  Cangiani, Andrea and Manzini, Gianmarco and Marini, L
                  Donatella and Russo, Alessandro},
  journal       = {Math. Models Methods Appl. Sci.},
  volume        = {23},
  number        = {01},
  pages         = {199--214},
  year          = {2013},
  publisher     = {World Scientific}
}

@Article{         Beirao-da-Veiga.Brezzi.ea:14,
  title         = {The hitchhiker's guide to the virtual element method},
  author        = {{Beir{\~a}o da Veiga}, L. and Brezzi, F. and Marini, L.D.
                  and Russo, A.},
  journal       = {Math. Models Methods Appl. Sci.},
  volume        = {24},
  number        = {08},
  pages         = {1541--1573},
  year          = {2014},
  doi           = {10.1142/S021820251440003X}
}

@Article{         Beirao-da-Veiga.Brezzi.ea:16,
  title         = {Virtual element method for general second-order elliptic
                  problems on polygonal meshes},
  author        = {{Beir\~ao da Veiga}, Lourenco and Brezzi, Franco and
                  Marini, Luisa Donatella and Russo, Alessandro},
  journal       = {Math. Models Methods Appl. Sci.},
  volume        = {26},
  number        = {04},
  pages         = {729--750},
  year          = {2016},
  publisher     = {World Scientific}
}

@Article{         Beirao-da-Veiga.Brezzi.ea:23,
  title         = {The virtual element method},
  author        = {{Beir\~ao da Veiga}, Louren\c{c}o and Brezzi, Franco and
                  Marini, L Donatella and Russo, Alessandro},
  journal       = {Acta Numerica},
  volume        = {32},
  pages         = {123--202},
  year          = {2023},
  doi           = {10.1017/S0962492922000095}
}

@Article{         Beirao-da-Veiga.Dassi.ea:22,
  title         = {Arbitrary-order pressure-robust {DDR} and {VEM} methods
                  for the {Stokes} problem on polyhedral meshes},
  author        = {{Beir\~ao da Veiga}, L. and Dassi, F. and Di Pietro, D. A.
                  and Droniou, J.},
  year          = {2022},
  journal       = {Comput. Meth. Appl. Mech. Engrg.},
  volume        = {397},
  number        = {115061},
  doi           = {10.1016/j.cma.2022.115061}
}

@Article{         Beirao-da-Veiga.Lipnikov.ea:11,
  title         = {Arbitrary-order nodal mimetic discretizations of elliptic
                  problems on polygonal meshes},
  author        = {{Beir\~ao da Veiga}, Lourenco and Lipnikov, Konstantin and
                  Manzini, Gianmarco},
  journal       = {SIAM J. Numer. Anal.},
  volume        = {49},
  number        = {5},
  pages         = {1737--1760},
  year          = {2011},
  publisher     = {SIAM}
}

@Book{            Beirao-da-Veiga.Lipnikov.ea:14,
  author        = {{Beir\~ao da Veiga}, L. and Lipnikov, K. and Manzini, G.},
  title         = {The mimetic finite difference method for elliptic
                  problems},
  series        = {MS\&A. Modeling, Simulation and Applications},
  volume        = {11},
  publisher     = {Springer, Cham},
  year          = {2014},
  doi           = {10.1007/978-3-319-02663-3}
}

@Article{         Beirao-da-Veiga.Lovadina.ea:17,
  title         = {Stability analysis for the virtual element method},
  author        = {{Beir\~ao da Veiga}, Louren\c{c}o and Lovadina, Carlo and
                  Russo, Alessandro},
  journal       = {Math. Models Methods Appl. Sci.},
  volume        = {27},
  number        = {13},
  pages         = {2557--2594},
  year          = {2017},
  doi           = {10.1142/S021820251750052X}
}

@Article{         Beirao-da-Veiga.Mascotto.ea:22,
  title         = {Interpolation and stability estimates for edge and face
                  virtual elements of general order},
  author        = {{Beir\~ao da Veiga}, Louren\c{c}o and Mascotto, Lorenzo
                  and Meng, Jian},
  journal       = {Math. Models Methods Appl. Sci.},
  volume        = {32},
  number        = {08},
  pages         = {1589--1631},
  year          = {2022},
  publisher     = {World Scientific}
}

@Article{         Beirao-da-Veiga.Vacca:22,
  title         = {Sharper error estimates for virtual elements and a
                  bubble-enriched version},
  author        = {{Beir\~ao da Veiga}, L and Vacca, Giuseppe},
  journal       = {SIAM J. Numer. Anal.},
  volume        = {60},
  number        = {4},
  pages         = {1853--1878},
  year          = {2022},
  doi           = {10.1137/21M1411275}
}

@Article{         Boffi.Di-Pietro:18,
  author        = {Boffi, D. and Di Pietro, D. A.},
  title         = {Unified formulation and analysis of mixed and primal
                  discontinuous skeletal methods on polytopal meshes},
  journal       = {ESAIM: Math. Model. Numer. Anal.},
  year          = {2018},
  volume        = {52},
  number        = {1},
  pages         = {1--28},
  doi           = {10.1051/m2an/2017036}
}

@Article{         Bonaldi.Di-Pietro.ea:25,
  title         = {An exterior calculus framework for polytopal methods},
  author        = {Bonaldi, F. and Di Pietro, D. A. and Droniou, J. and Hu,
                  K.},
  journal       = {J. Eur. Math. Soc.},
  year          = {2025},
  doi           = {10.4171/JEMS/1602},
  note          = {Published online}
}

@Article{         Bonelle.Ern:14,
  title         = {Analysis of compatible discrete operator schemes for
                  elliptic problems on polyhedral meshes},
  author        = {Bonelle, J. and Ern, A.},
  journal       = {ESAIM: Math. Model. Numer. Anal.},
  volume        = {48},
  pages         = {553--581},
  year          = {2014},
  doi           = {10.1051/m2an/2013104}
}

@Article{         Botti.Di-Pietro.ea:19,
  author        = {Botti, M. and Di Pietro, D. A. and Guglielmana, A.},
  title         = {A low-order nonconforming method for linear elasticity on
                  general meshes},
  journal       = {Comput. Meth. Appl. Mech. Engrg.},
  year          = {2019},
  volume        = {354},
  pages         = {96--118},
  doi           = {10.1016/j.cma.2019.05.031}
}

@Article{         Brenner.Guan.ea:17,
  author        = {Brenner, S. C. and Guan, Q. and Sung, L.-Y.},
  title         = {Some estimates for {Virtual Element} methods},
  journal       = {Comput. Methods Appl. Math.},
  year          = {2017},
  volume        = {17},
  number        = {4},
  pages         = {553–-574}
}

@Article{         Brenner.Sung.ea:13,
  title         = {A {Morley} finite element method for the displacement
                  obstacle problem of clamped Kirchhoff plates},
  journal       = {Journal of Computational and Applied Mathematics},
  volume        = {254},
  pages         = {31-42},
  year          = {2013},
  doi           = {10.1016/j.cam.2013.02.028},
  author        = {Brenner, S. C. and Sung, Li-Yeng and Zhang, Hongchao and
                  Zhang, Yi}
}

@Article{         Brenner.Sung:18,
  title         = {Virtual element methods on meshes with small edges or
                  faces},
  author        = {Brenner, Susanne C and Sung, Li-Yeng},
  journal       = {Math. Models Methods Appl. Sci.},
  volume        = {28},
  number        = {07},
  pages         = {1291--1336},
  year          = {2018},
  publisher     = {World Scientific}
}

@Article{         Brenner:99,
  author        = {Brenner, Susanne C.},
  title         = {Convergence of nonconforming multigrid methods without
                  full elliptic regularity},
  journal       = {Math. Comp.},
  volume        = {68},
  year          = {1999},
  number        = {225},
  pages         = {25--53},
  doi           = {10.1090/S0025-5718-99-01035-2}
}

@Article{         Brezzi.Buffa.ea:09,
  title         = {Mimetic finite differences for elliptic problems},
  author        = {Brezzi, Franco and Buffa, Annalisa and Lipnikov,
                  Konstantin},
  journal       = {ESAIM: Mathematical Modelling and Numerical Analysis},
  volume        = {43},
  number        = {2},
  pages         = {277--295},
  year          = {2009},
  publisher     = {EDP Sciences}
}

@Article{         Brezzi.Lipnikov.ea:05,
  title         = {Convergence of the mimetic finite difference method for
                  diffusion problems on polyhedral meshes},
  author        = {Brezzi, Franco and Lipnikov, Konstantin and Shashkov,
                  Mikhail},
  journal       = {SIAM Journal on Numerical Analysis},
  volume        = {43},
  number        = {5},
  pages         = {1872--1896},
  year          = {2005},
  publisher     = {SIAM}
}

@Article{         Cangiani.Manzini.ea:17,
  title         = {Conforming and nonconforming virtual element methods for
                  elliptic problems},
  author        = {Cangiani, Andrea and Manzini, Gianmarco and Sutton, Oliver
                  J},
  journal       = {IMA J. Numer. Anal.},
  volume        = {37},
  number        = {3},
  pages         = {1317--1354},
  year          = {2017},
  publisher     = {Oxford University Press}
}

@Article{         Carstensen.Puttkammer:20,
  author        = {Carstensen, Carsten and Puttkammer, Sophie},
  title         = {How to prove the discrete reliability for nonconforming
                  finite element methods},
  journal       = {J. Comput. Math.},
  volume        = {38},
  year          = {2020},
  number        = {1},
  pages         = {142--175},
  doi           = {10.4208/jcm.1908-m2018-0174}
}

@Article{         Cea:64,
  author        = {C\'ea, Jean},
  title         = {Approximation variationnelle des probl\`emes aux limites},
  journal       = {Ann. Inst. Fourier (Grenoble)},
  volume        = {14},
  year          = {1964},
  pages         = {345--444},
  url           = {http://www.numdam.org/item?id=AIF_1964__14_2_345_0}
}

@Article{         Chen.Huang:18,
  title         = {Some error analysis on virtual element methods},
  author        = {Chen, Long and Huang, Jianguo},
  journal       = {Calcolo},
  volume        = {55},
  number        = {1},
  pages         = {5},
  year          = {2018},
  publisher     = {Springer}
}

@Article{         Christiansen.Munthe-Kaas.ea:11,
  author        = {Christiansen, Snorre H. and Munthe-Kaas, Hans Z. and
                  Owren, Brynjulf},
  title         = {Topics in structure-preserving discretization},
  journal       = {Acta Numerica},
  volume        = {20},
  year          = {2011},
  pages         = {1--119},
  doi           = {10.1017/S096249291100002X}
}

@Misc{            Christiansen.Rapetti:25,
  author        = {Christiansen, S. H. and Rapetti, F.},
  title         = {Interpretation of a {Discrete de Rham} method as a {Finite
                  Element System}},
  year          = {2025},
  month         = {12},
  eprint        = {2512.05912},
  archiveprefix = {arXiv},
  primaryclass  = {math.NA}
}

@Article{         Christiansen:08,
  author        = {Christiansen, Snorre H.},
  title         = {A construction of spaces of compatible differential forms
                  on cellular complexes},
  journal       = {Math. Models Methods Appl. Sci.},
  volume        = {18},
  year          = {2008},
  number        = {5},
  pages         = {739--757},
  doi           = {10.1142/S021820250800284X}
}

@Article{         Di-Pietro.Droniou.ea:20,
  author        = {Di Pietro, D. A. and Droniou, J. and Rapetti, F.},
  title         = {Fully discrete polynomial {de Rham} sequences of arbitrary
                  degree on polygons and polyhedra},
  journal       = {Math. Models Methods Appl. Sci.},
  year          = {2020},
  volume        = {30},
  number        = {9},
  pages         = {1809-1855},
  doi           = {10.1142/S0218202520500372}
}

@Article{         Di-Pietro.Droniou.ea:24,
  title         = {A pressure-robust {Discrete de Rham} scheme for the
                  {Navier--Stokes} equations},
  author        = {Di Pietro, D. A. and Droniou, J. and Qian, J. J.},
  journal       = {Comput. Meth. Appl. Mech. Engrg.},
  year          = {2024},
  volume        = {421},
  number        = {116765},
  doi           = {10.1016/j.cma.2024.116765}
}

@Misc{            Di-Pietro.Droniou.ea:25,
  author        = {Di Pietro, D. A. and Droniou, J. and Pitassi, S.},
  title         = {Conforming lifting and adjoint consistency for the
                  {Discrete de Rham} complex of differential forms},
  year          = {2025},
  month         = {9},
  eprint        = {2509.21449},
  archiveprefix = {arXiv},
  primaryclass  = {math.NA}
}

@Article{         Di-Pietro.Droniou:18,
  author        = {Di Pietro, D. A. and Droniou, J.},
  title         = {A third {Strang} lemma for schemes in fully discrete
                  formulation},
  year          = {2018},
  journal       = {Calcolo},
  volume        = {55},
  number        = {40},
  doi           = {10.1007/s10092-018-0282-3}
}

@Book{            Di-Pietro.Droniou:20,
  author        = {Di Pietro, D. A. and Droniou, J.},
  title         = {The {Hybrid High-Order} method for polytopal meshes},
  subtitle      = {Design, analysis, and applications},
  publisher     = {Springer International Publishing},
  year          = {2020},
  series        = {Modeling, Simulation and Application},
  number        = {19},
  doi           = {10.1007/978-3-030-37203-3}
}

@Article{         Di-Pietro.Droniou:22,
  author        = {Di Pietro, Daniele A. and Droniou, J\'{e}r\^{o}me},
  title         = {A discrete de {R}ham method for the {R}eissner-{M}indlin
                  plate bending problem on polygonal meshes},
  journal       = {Comput. Math. Appl.},
  volume        = {125},
  year          = {2022},
  pages         = {136--149},
  doi           = {10.1016/j.camwa.2022.08.041}
}

@Article{         Di-Pietro.Droniou:23,
  author        = {Di Pietro, D. A. and Droniou, J.},
  title         = {An arbitrary-order discrete {de Rham} complex on
                  polyhedral meshes: Exactness, {Poincar\'e} inequalities,
                  and consistency},
  journal       = {Found. Comput. Math.},
  year          = {2023},
  volume        = {23},
  pages         = {85--164},
  doi           = {10.1007/s10208-021-09542-8}
}

@Article{         Di-Pietro.Ern.ea:14,
  author        = {Di Pietro, D. A. and Ern, A. and Lemaire, S.},
  title         = {An arbitrary-order and compact-stencil discretization of
                  diffusion on general meshes based on local reconstruction
                  operators},
  journal       = {Comput. Meth. Appl. Math.},
  year          = {2014},
  volume        = {14},
  number        = {4},
  pages         = {461--472},
  doi           = {10.1515/cmam-2014-0018}
}

@Article{         Di-Pietro.Ern:15,
  author        = {Di Pietro, D. A. and Ern, A.},
  title         = {A hybrid high-order locking-free method for linear
                  elasticity on general meshes},
  journal       = {Comput. Meth. Appl. Mech. Engrg.},
  year          = {2015},
  volume        = {283},
  pages         = {1--21},
  doi           = {10.1016/j.cma.2014.09.009}
}

@Article{         Droniou.Eymard.ea:10,
  author        = {Droniou, J\'er\^ome and Eymard, Robert and Gallou\"et,
                  Thierry and Herbin, Rapha\`ele},
  title         = {A unified approach to mimetic finite difference, hybrid
                  finite volume and mixed finite volume methods},
  journal       = {Math. Models Methods Appl. Sci.},
  volume        = {20},
  year          = {2010},
  number        = {2},
  pages         = {265--295},
  doi           = {10.1142/S0218202510004222}
}

@Article{         Droniou.Eymard.ea:16,
  author        = {Droniou, J{\'e}r{\^o}me and Eymard, Robert and Herbin,
                  Rapha\`ele},
  title         = {Gradient schemes: generic tools for the numerical analysis
                  of diffusion equations},
  journal       = {M2AN Math. Model. Numer. Anal.},
  volume        = {50},
  year          = {2016},
  number        = {3},
  pages         = {749--781},
  note          = {Special issue -- Polyhedral discretization for PDE},
  doi           = {10.1051/m2an/2015079}
}

@Book{            Droniou.Eymard.ea:18,
  title         = {The gradient discretisation method},
  author        = {Droniou, J\'er\^ome and Eymard, Robert and Gallou\"et,
                  Thierry and Guichard, Cindy and Herbin, Rapha\`ele},
  year          = {2018},
  pages         = {511p},
  series        = {Mathematics \& Applications},
  publisher     = {Springer},
  volume        = {82},
  isbn          = {978-3-319-79041-1 (Softcover) 978-3-319-79042-8 (eBook)},
  doi           = {10.1007/978-3-319-79042-8}
}

@Article{         Droniou.Eymard:06,
  author        = {Droniou, J\'er\^ome and Eymard, Robert},
  title         = {A mixed finite volume scheme for anisotropic diffusion
                  problems on any grid},
  journal       = {Numer. Math.},
  volume        = {105},
  year          = {2006},
  number        = {1},
  pages         = {35--71},
  doi           = {10.1007/s00211-006-0034-1}
}

@Article{         Droniou.Haidar.ea:25,
  author        = {Droniou, J\'er\^ome and Haidar, Ali and Masson, Roland},
  title         = {Analysis of a VEM--fully discrete polytopal scheme with
                  bubble stabilisation for contact mechanics with Tresca
                  friction},
  journal       = {ESAIM: M2AN Math. Model. Numer. Anal.},
  year          = {2025},
  volume        = {59},
  number        = {2},
  pages         = {1043-1074},
  doi           = {10.1051/m2an/2025013}
}

@Article{         Ern.Guermond:16,
  author        = {Ern, A. and Guermond, J.-L.},
  title         = {Mollification in strongly {L}ipschitz domains with
                  application to continuous and discrete de {R}ham
                  complexes},
  journal       = {Comput. Methods Appl. Math.},
  volume        = {16},
  year          = {2016},
  number        = {1},
  pages         = {51--75},
  doi           = {10.1515/cmam-2015-0034}
}

@Book{            Ern.Guermond:21*1,
  author        = {Ern, Alexandre and Guermond, Jean-Luc},
  title         = {Finite elements {II}},
  subtitle      = {Approximation and interpolation},
  series        = {Texts in Applied Mathematics},
  volume        = {73},
  publisher     = {Springer, Cham},
  year          = {2021},
  doi           = {10.1007/978-3-030-56923-5}
}

@Article{         Ern.Vohralik:15,
  author        = {Ern, A. and Vohral\'{i}k, M.},
  title         = {Polynomial-degree-robust a posteriori estimates in a
                  unified setting for conforming, nonconforming,
                  discontinuous {Galerkin}, and mixed discretizations},
  journal       = {SIAM J. Numer. Anal.},
  year          = {2015},
  volume        = {53},
  number        = {2},
  pages         = {1058--1081},
  doi           = {10.1137/130950100}
}

@Article{         Ern.Vohralik:20,
  author        = {Ern, A. and Vohral\'{i}k, M.},
  title         = {Stable broken $H^1$ and
                  $\boldsymbol{H}(\operatorname{div})$ polynomial extensions
                  for polynomial-degree-robust potential and flux
                  reconstruction in three space dimensions},
  journal       = {Math. Comp.},
  year          = {2020},
  volume        = {89},
  number        = {322},
  pages         = {551--594},
  doi           = {10.1090/mcom/3482}
}

@InCollection{    Eymard.Gallouet.ea:00,
  author        = {Eymard, R. and Gallou{\"e}t, T. and Herbin, R.},
  title         = {Finite volume methods},
  booktitle     = {Techniques of Scientific Computing, Part III},
  series        = {Handbook of Numerical Analysis, VII},
  pages         = {713--1020},
  publisher     = {North-Holland},
  editor        = {Ciarlet, P. G. and Lions, J.-L.},
  address       = {Amsterdam},
  year          = {2000}
}

@Article{         Karakashian.Pascal:03,
  author        = {Karakashian, Ohannes A. and Pascal, Frederic},
  title         = {A posteriori error estimates for a discontinuous
                  {G}alerkin approximation of second-order elliptic
                  problems},
  journal       = {SIAM J. Numer. Anal.},
  volume        = {41},
  year          = {2003},
  number        = {6},
  pages         = {2374--2399},
  doi           = {10.1137/S0036142902405217}
}

@Article{         Kuznetsov.Lipnikov.ea:04,
  title         = {The mimetic finite difference method on polygonal meshes
                  for diffusion-type problems},
  author        = {Kuznetsov, Yuri and Lipnikov, Konstantin and Shashkov,
                  Mikhail},
  journal       = {Comput. Geosci.},
  volume        = {8},
  number        = {4},
  pages         = {301--324},
  year          = {2004},
  publisher     = {Springer}
}

@Article{         Linke:14,
  author        = {Linke, Alexander},
  title         = {On the role of the {H}elmholtz decomposition in mixed
                  methods for incompressible flows and a new variational
                  crime},
  journal       = {Comput. Methods Appl. Mech. Engrg.},
  volume        = {268},
  year          = {2014},
  pages         = {782--800},
  doi           = {10.1016/j.cma.2013.10.011}
}

@Article{         Mallik.Nataraj:16,
  author        = {Mallik, Gouranga and Nataraj, Neela},
  title         = {A nonconforming finite element approximation for the von
                  {K}arman equations},
  journal       = {ESAIM Math. Model. Numer. Anal.},
  volume        = {50},
  year          = {2016},
  number        = {2},
  pages         = {433--454},
  doi           = {10.1051/m2an/2015052}
}

@Article{         Mascotto:23,
  title         = {The role of stabilization in the virtual element method: a
                  survey},
  author        = {Mascotto, Lorenzo},
  journal       = {Comput. Math. Appl.},
  volume        = {151},
  pages         = {244--251},
  year          = {2023},
  publisher     = {Elsevier}
}

@Article{         Mora.Rivera.ea:15,
  title         = {A virtual element method for the Steklov eigenvalue
                  problem},
  author        = {Mora, David and Rivera, Gonzalo and Rodr{\'\i}guez,
                  Rodolfo},
  journal       = {Math. Models Methods Appl. Sci.},
  volume        = {25},
  number        = {08},
  pages         = {1421--1445},
  year          = {2015},
  publisher     = {World Scientific}
}

@InBook{          Strang:72,
  author        = {Strang, G.},
  title         = {The Mathematical Foundations of the Finite Element Method
                  with Applications to Partial Diﬀerential Equations},
  chapter       = {Variational crimes in the finite element method},
  editor        = {A. Aziz},
  publisher     = {Academic Press},
  address       = {New York, NY},
  year          = {1972}
}

@Misc{            Vohralik:24,
  title         = {A posteriori numerical analysis based on the method of
                  equilibrated fluxes},
  author        = {Vohral\'{i}k, M.},
  month         = {4},
  year          = {2024},
  url           = {https://who.rocq.inria.fr/Martin.Vohralik/Enseig/APost/a_posteriori.pdf},
  note          = {Lecture notes for the course NMNV464, Charles University,
                  Prague}
}

@Book{            Wriggers.Aldakheel.ea:24,
  title         = {Virtual element methods in engineering sciences},
  author        = {Wriggers, Peter and Aldakheel, Fadi and Hudobivnik,
                  Bla{\v{z}}},
  year          = {2024},
  publisher     = {Springer}
}

\end{document}